\documentclass[review]{elsarticle}

\usepackage{lineno,hyperref}

\journal{journal}
\date{}

\graphicspath{{../Figures/},{./Figures/},{./Figures/specialRE/},{./Figures/Prelude_step/},{./Figures/Prelude_gamma/},{./Figures/Nicholson/}}

\usepackage{amsmath,amsfonts,amssymb,amsthm}
\usepackage{booktabs}
\usepackage{pgfplots}
\pgfplotsset{compat=1.10}

\usepackage{mathrsfs} 

\newtheorem{theorem}{Theorem}
\newtheorem{lemma}[theorem]{Lemma}
\newtheorem{proposition}[theorem]{Proposition}
\newdefinition{remark}{Remark}

\counterwithin*{equation}{section}

\newcommand{\LL}{\mathcal{L}} 
\newcommand*{\diff}{\mathop{}\!\mathrm{d}}
\newcommand*{\ee}{\mathop{}\!\mathrm{e}} 
\renewcommand*{\Re}{\mathop{}\!\mathrm{Re}\,} 

\definecolor{myblue}{rgb}{0.00000,0.44700,0.74100}%
\newcommand{\fra}[1]{\textcolor{black}{#1}}

\begin{document}

\begin{frontmatter}

\title{Numerical bifurcation analysis of renewal equations via pseudospectral approximation
\tnoteref{mytitlenote}}
\tnotetext[mytitlenote]{Declarations of interest: none.}


\author[Manchester,York,UD]{Francesca Scarabel\corref{mycorrespondingauthor}
}
\cortext[mycorrespondingauthor]{Corresponding author}
\ead{francesca.scarabel@manchester.ac.uk}

\author[Utrecht]{Odo Diekmann}
\ead{O.Diekmann@uu.nl}

\author[UD]{Rossana Vermiglio}
\ead{rossana.vermiglio@uniud.it}

\address[Manchester]{Department of Mathematics, The University of Manchester, Oxford Rd, M13 9PL Manchester, UK}
\address[York]{LIAM -- Laboratory for Industrial and Applied Mathematics, Department of Mathematics and Statistics, York University, 4700 Keele Street, Toronto, ON M3J 1P3, Canada}
\address[UD]{CDLab -- Computational Dynamics Laboratory, Department of Mathematics, Computer Science and Physics, University of Udine, via delle scienze 206, 33100 Udine, Italy}
\address[Utrecht]{Department of Mathematics, Utrecht University, P.O. Box 80010, 3508 TA Utrecht, The Netherlands}

\begin{abstract}
We propose an approximation of nonlinear renewal equations by means of ordinary differential equations.
We consider the integrated state, which is absolutely continuous and satisfies a delay differential equation. By applying the pseudospectral approach to the abstract formulation of the differential equation, we obtain an approximating system of ordinary differential equations.
We present convergence proofs for equilibria and the associated characteristic roots, and we use some models from ecology and epidemiology to illustrate the benefits of the approach to perform numerical bifurcation analyses of equilibria and periodic solutions.
The numerical simulations show that the implementation of the new approximating system can be about ten times more efficient in terms of computational times than the one originally proposed in [Breda et al, SIAM Journal on Applied Dynamical Systems, 2016], as it avoids the numerical inversion of an algebraic equation.
\end{abstract}

\begin{keyword}
nonlinear renewal equation; equilibria; periodic solutions; Hopf bifurcation; stability analysis; pseudospectral method
\end{keyword}

\end{frontmatter}


%

\section{Introduction}

In our perspective, a renewal equation (RE) is a delay equation, i.e., a rule for extending a function of time towards the future on the basis of the (assumed to be) known history. The difference between delay differential equations (DDE) and RE is that, for the former, the rule specifies the derivative of the function in the current time point, while for the latter the rule specifies the function value itself.
For both, one defines a dynamical system by translation along the extended function, so by updating the history. 
But while for DDE the natural state space consists of continuous functions (with the supremum norm), for RE it is more natural to consider integrable functions (with the $L^1$ norm). 
Stability and bifurcation results for RE are formulated in \cite{DGG2007} as corollaries of the general theory presented in \cite{DiekmannBook}.

RE arise routinely in the formulation of physiologically structured population models, see e.g. \cite{DGM2010}. In that context, one would often like to go beyond a pen and paper analysis and perform a numerical bifurcation analysis. But the lack of tools that can handle this kind of delay equations clearly forms an obstruction.

In \cite{SIADS2016} the idea is launched to first reduce the infinite dimensional dynamical system corresponding to a delay equation to a finite dimensional one by pseudospectral approximation, and next use tools for ordinary differential equations (ODE) in order to perform a numerical bifurcation analysis. Several examples illustrate that this approach is promising (also see \cite{EJQTDE2016, Ando2019b, Babette, DCDS2020Vermiglio, JMB2019, Vietnam2020, AMC2018}).
Note that we restrict to \emph{bounded} maximal delay, as unbounded delays require to consider different (exponentially weighted) function spaces and corresponding results from weighted interpolation theory \cite{AMC2018,DG2012}.

A nice feature of pseudospectral approximation is that, in the resulting ODE, one can recognize that the dynamics involves both a rule for extension and translation. The latter is captured by a matrix, often called differentiation matrix, that depends on the choice of mesh points but \emph{not} on the delay equation under consideration. 
\fra{In the approximation of }
a scalar DDE, the rule for extension is reflected in the expression for the derivative of exactly one component, so the nonlinear part of the ODE has one dimensional range; 
\fra{in the approximation of }
systems of DDE, the dimension of the range corresponds to the dimension of the system. The fact that, for RE, the rule for extension does specify the value, rather than the derivative, makes its incorporation in the ODE less straightforward. In \cite{SIADS2016,EJQTDE2016} an ad hoc method was employed: the value in the current time point was computed from the (approximate) history and the right-hand side of the RE by way of a numerical solver. The aim of the present paper is to introduce a much more natural and elegant alternative, which also improves the efficiency of the numerical method.

The main new idea is to approximate the indefinite integral of the integrable function, rather than the integrable function itself. First of all, a reassuring consequence is that now we approximate a function that has well-defined point values (in contrast with an $L^1$ equivalence class).
More importantly: in terms of the integrated function, the original RE becomes a DDE, as the rule for extension specifies its derivative in the current time point. 
The difference with a ``true'' DDE is that we have to incorporate a (re)normalization condition in order to have a one-to-one relationship between the integrated function and its derivative. The resulting ODE therefore has a slightly different structure: for a scalar equation the nonlinear part again has one-dimensional range, but in natural coordinates the range is spanned by a different vector.

From a more abstract point of view, we \fra{represent $L^1$ by $AC_0$}, the subspace of $NBV$ (normalized bounded variation functions) consisting of absolutely continuous functions. 
This embedding also features in the \emph{sun-star} framework of \cite{DGG2007,DiekmannBook} (where the ``big'' space $NBV$ serves to represent the rule for extension as a perturbation with range spanned by a Dirac mass) and in the more recent theory of twin semigroups \cite{TwinSemigroups}. 
The space \fra{$AC_0$} also guarantees the convergence of the approximation: classical results from interpolation theory \cite{Krylov, MastroianniBook, Trefethen2013} ensure the convergence of polynomial interpolation, in supremum norm (and not necessarily in $NBV$ norm), for functions that are at least absolutely continuous, for a suitable choice of the interpolation nodes in the bounded interval. 

An important advantage of the current method compared to the approach proposed in \cite{SIADS2016} is the remarkable reduction of computational costs in all the simulations considered here (see for instance Figure \ref{f:specialRE_time} and Table \ref{t:times} below).
The inversion of the nonlinear condition with a numerical solver is indeed the main bottleneck of the method in \cite{SIADS2016}.
The improved computational efficiency is fundamental especially when dealing with complex applications from population dynamics as the coupled RE/DDE models for \emph{Daphnia} \cite{DGM2010}, that are particularly challenging to treat numerically and often need ad hoc techniques \cite{Vietnam2020, Ando2020a, DCDS2020Ando, PSPManalysis2020}.

\smallskip
In this paper we focus on the approximation of equilibria and their stability. 
We start in the next section by providing some concrete examples that illustrate the main features and potential of the approximation approach.
The latter is introduced rigorously in Section \ref{s:method} for general scalar nonlinear RE, together with some basic results regarding the approximation of equilibria.
Section \ref{s:linear_conv} focuses on autonomous linear equations: we show that characteristic roots and exponential solutions are approximated with infinite order of convergence as the dimension of the approximation increases.
An outlook discussion is presented in Section \ref{s:outlook}.

In population models the RE usually concerns the population level birth rate and for that reason we chose to use the character $b$ to denote the variable. The integrated quantity corresponds to the cumulative birth rate and is denoted by $B$.

\section{Some illustrative examples}
\label{s:examples}

In this section we consider some specific nonlinear RE.
All the equations are approximated with an ODE system using the method introduced rigorously in Section \ref{s:method}. The dimension $M$ of the approximating ODE system is specified each time. 
The bifurcation diagram of the approximating system is then studied numerically using software for numerical bifurcation analysis of ODE.
Specifically we use the package MatCont (version 7p1) \cite{Matcont0} running on MATLAB 2019a.
To improve efficiency, the integrals are computed using Clenshaw--Curtis quadrature formulas \cite{Trefethen2000}.
\fra{MATLAB codes used to obtain the results in this paper are available at \href{http://cdlab.uniud.it/software}{\path{http://cdlab.uniud.it/software}}.
}

We have three main goals: 
1) show the suitability of the approach (ODE approximation plus software for numerical bifurcation) to reveal bifurcations of equilibria and to study some more advanced dynamical behaviors; 2) carry out a preliminary study of the convergence of the approximations; 
3) compare the performances and the output of the method presented here with the one proposed in \cite{SIADS2016}.

\subsection{An SIRS model}

Let $k \colon \mathbb{R} \to \mathbb{R}$ be a nonnegative and measurable function with support in $[0,1]$, and normalized such that $\int_0^1 k(s)\diff s=1$. 
Consider the nonlinear equation
\begin{equation} \label{prelude}
b(t) = \gamma \left(1-\int_0^1 b(t-s)\diff s \right) \int_0^1 k(s) b(t-s) \diff s, \qquad t>0,
\end{equation}
for \fra{$\gamma > 0$}. 
In \cite{Diekmann1982prelude}, the authors derive this equation in the context of an SIRS epidemic model, and study the bifurcation with respect to $\gamma$. It is proved that the trivial equilibrium undergoes a transcritical bifurcation at $\gamma=1$, and a positive stable equilibrium exists for $\gamma>1$.
Under some conditions on the kernel $k$, the positive equilibrium undergoes a sequence of Hopf bifurcations as $\gamma$ increases.
The authors also conjecture that \eqref{prelude} may exhibit chaotic behavior for large values of $\gamma$.

We here consider a truncated Gamma-type kernel of the form
\begin{equation} \label{gamma_kernel}
k(s) =
\begin{cases}
 \alpha s^{m-1} \ee^{-\frac{s}{\theta}} & s \leq 1, \\
 0 & s > 1,
 \end{cases}
\end{equation}
for fixed parameters $m=3$, $\theta=0.1$, and $\alpha$ a normalising constant so that $\int_0^1 k(s) \diff s=1$. 
\fra{The kernel \eqref{gamma_kernel} belongs to the class considered in \cite{Diekmann1982prelude} due to the cut-off at $s=1$.}
The bifurcation diagram with respect to $\log \gamma$ is shown in Figure \ref{f:prelude_gamma_bif} for different values of the dimension $M$ of the approximating system. 
A Hopf bifurcation is detected on the branch of positive equilibria, after which the equilibrium becomes unstable and a branch of stable periodic solutions appears.
Note that the minimum values of the periodic orbits come close to zero when $\log \gamma$ approaches 3. When this happens, $M = 10$ is not sufficient to approximate the orbit, but for $M=20,40$ the lower portions of the curves are indistinguishable.

\begin{figure}[t]
\centering
\includegraphics[scale=1]{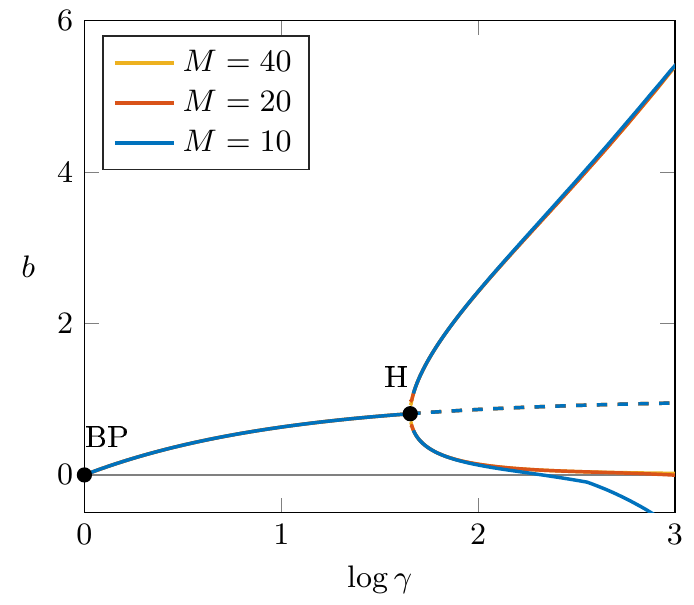}
\caption{Equation \eqref{prelude} with kernel \eqref{gamma_kernel}, with $m=3$ and $\theta=0.1$. Bifurcation diagram with respect to $\log \gamma$, for different values of $M$. A Hopf bifurcation is detected at $\log\gamma \approx 1.6553$. Equilibrium branch and the maximum and minimum values of the periodic orbits are plotted; solid lines correspond to stable branches, dashed lines to unstable ones.
}
\label{f:prelude_gamma_bif}
\end{figure}

\subsection{Nicholson's blowflies equation} 
Consider Nicholson's blowflies DDE \cite{Gurney1980}
 \begin{equation} \label{Nich-DDE}
 A'(t) = -\mu A(t) + \beta_0 \ee^{-\mu} A(t-1) \ee^{-A(t-1)}, \qquad t>0,
 \end{equation}
for $\beta_0,\mu \geq 0$, where $A(t)$ denotes the size of the adult population, and newborns become adult after a maturation delay which is normalized to 1.
For $b(t):= \beta_0 A(t) \ee^{-A(t)},$ \eqref{Nich-DDE} is equivalent to
\begin{equation}\label{Nich-RE}
b(t) = \beta_0 \ee^{-\int_1^{\infty} b(t-s) \ee^{-\mu s} \diff s} \, \int_1^{\infty} b(t-s) \ee^{-\mu s} \diff s, \qquad t>0.
\end{equation}
The equivalence is rigorously proved in \ref{app:blowflies}; however note that, since $b(t)$ represents the population birth rate, \eqref{Nich-RE} follows directly from modelling assumptions and there is no need to derive it from the DDE.

It is easy to see that the (unique) nontrivial equilibrium of \eqref{Nich-DDE} exists for $\beta_0 > \mu \ee^{\mu}$, with value
$$ \overline A = \log \frac{\beta_0}{\mu} - \mu.$$ 
Correspondingly, the equilibrium of \eqref{Nich-RE} is $\overline b = (\log \frac{\beta_0}{\mu} - \mu) \mu \ee^{\mu}$.

The nontrivial equilibrium undergoes a Hopf bifurcation as $\beta_0$ increases. 
An explicit parametrization of the Hopf bifurcation curves in the plane $(\mu,\beta_0 \ee^{-\mu}/\mu)$ is computed in \cite{Babette} for the DDE \eqref{Nich-DDE}.
Moreover, the numerical bifurcation analysis of \eqref{Nich-DDE} can be performed with standard software for DDE, like DDE-BIFTOOL for MATLAB \cite{ddebiftool0, ddebiftool}. 
Therefore, the blowflies equation allows us to compare the output of MatCont computations on the pseudospectral approximation of \eqref{Nich-RE} with analytic formulas for the Hopf bifurcation curves and with the output of DDE-BIFTOOL for bifurcations of periodic solutions of \eqref{Nich-DDE}. 

In order to apply the method presented here, we truncated the integral in \eqref{Nich-RE} so that the probability of survival at a finite maximal time $\tau$ is less than a certain threshold. We chose $\tau=10$, as it ensures that the survival probability $\ee^{-\mu \tau}$ is less than $10^{-6}$ 
for $\mu = 1$, less than $10^{-10}$ 
for $\mu=2$, and less than $10^{-19}$ 
for $\mu=4$.
Figure \ref{f:Nich-fixed-tau} shows the Hopf bifurcation curve approximated with different values of $M$, for fixed $\tau=10$, and compared with the analytic curve \fra{obtained for the DDE}. The curves in a two-parameter plane were obtained by first performing a one-parameter continuation for $\mu=4$, and then starting the two-parameter continuation from the detected Hopf bifurcation. Note that no Hopf bifurcation was detected for $M\leq 8$.

While in Figure \ref{f:Nich-fixed-tau} we fixed the delay interval $[-\tau,0]$ and varied the dimension $M$ of the approximation, Figure \ref{f:Nich-fixed-M} shows the effect of changing the truncation delay $\tau$ while fixing $M$. The results highlight the delicate balance between the choice of the truncation delay and the degree of \fra{the} approximation: for a good approximation, a larger interval requires a larger dimension of the approximating system. This is particularly evident in the analysis of complex bifurcations, for instance the period doubling bifurcation of periodic orbits, and demonstrates the need to develop an approximation specifically tailored to unbounded intervals.

\begin{figure}
\centering
\includegraphics[scale=1]{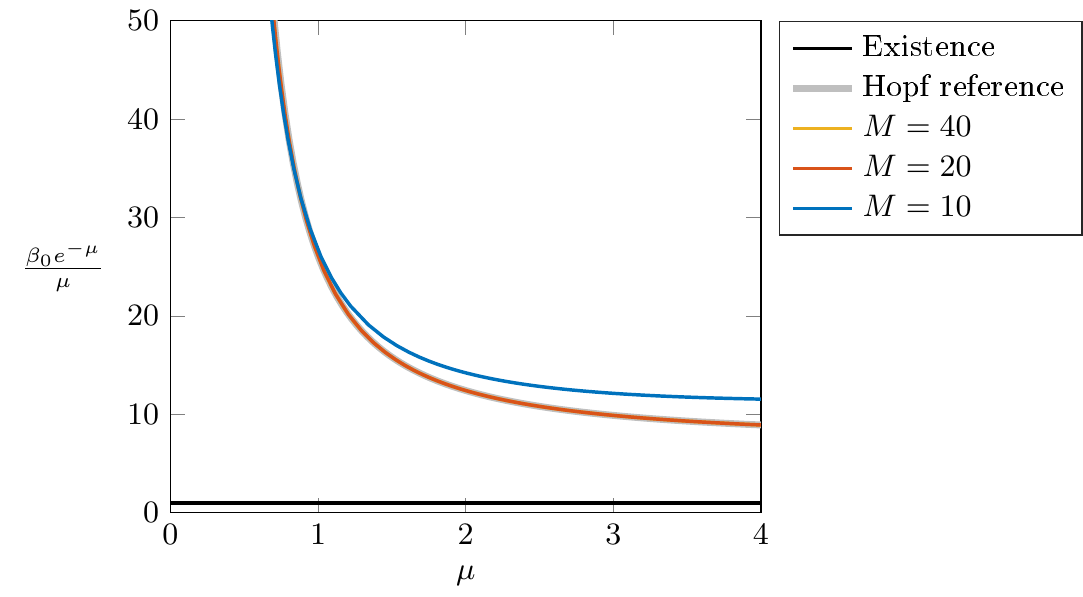}
\caption{Nicholson's blowflies model \eqref{Nich-RE}: Hopf bifurcation curve in the plane $(\mu,\beta_0 \ee^{-\mu}/\mu)$ approximated with MatCont, for fixed $\tau=10$ and different values of $M$. Note that the Hopf bifurcation curves for $M=20, 40$ are indistinguishable from each other and \fra{lie on top of the reference Hopf bifurcation curve calculated analytically in  \cite{Babette}}. No Hopf bifurcation was detected using $M\leq 8$.}
\label{f:Nich-fixed-tau}
\end{figure}

\begin{figure}
\centering
\includegraphics[scale=1]{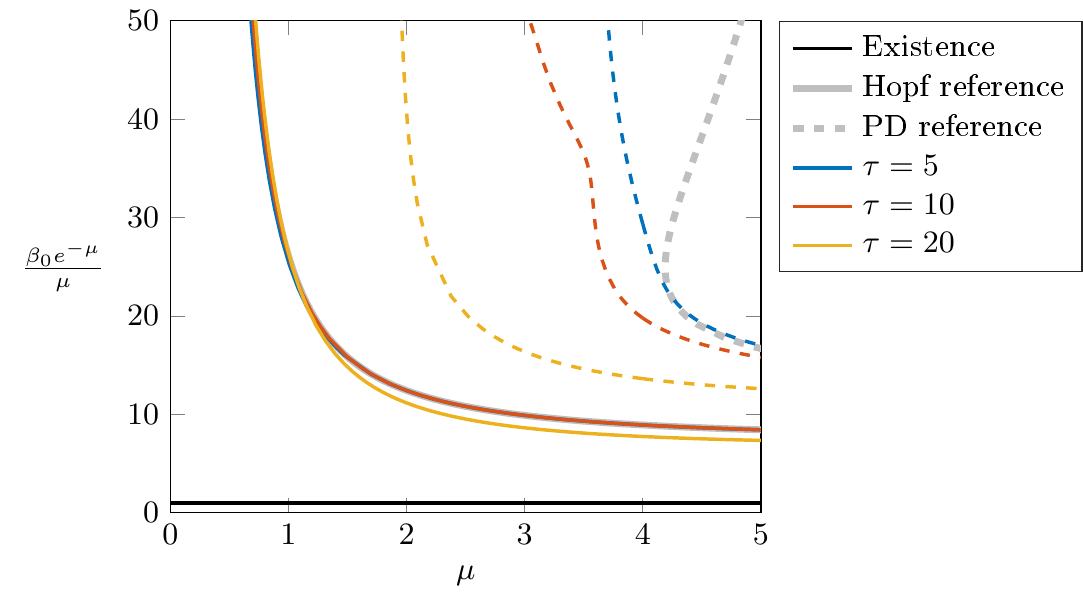}
\caption{Nicholson's blowflies model \eqref{Nich-RE}: Hopf and period doubling bifurcation curves in the plane $(\mu,\beta_0 \ee^{-\mu}/\mu)$ approximated with MatCont, for fixed $M=20$ and different values of $\tau$. The reference line for the Hopf bifurcation curve is calculated analytically in \cite{Babette}; the reference line for the period doubling bifurcation is approximated with DDE-BIFTOOL on the DDE formulation \eqref{Nich-DDE} of the model.}
\label{f:Nich-fixed-M}
\end{figure}

\subsection{A cannibalism equation}

\begin{figure}[p]
\centering
\includegraphics[width=.48\textwidth]{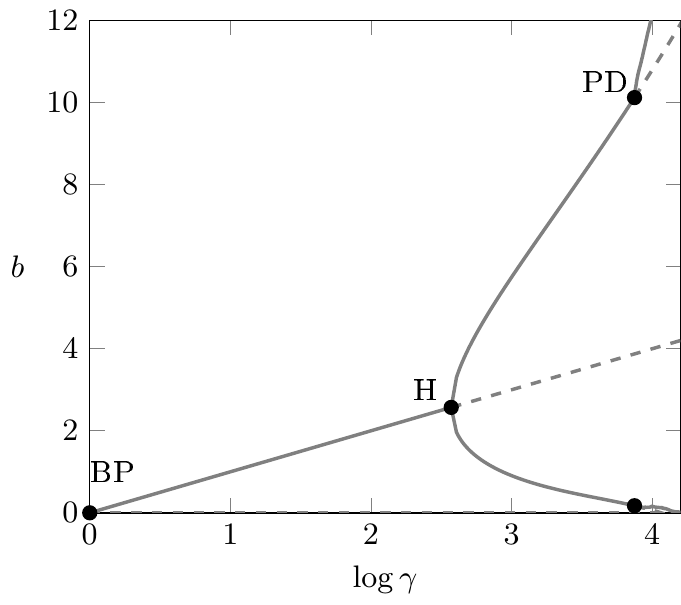}\quad
\includegraphics[width=.48\textwidth]{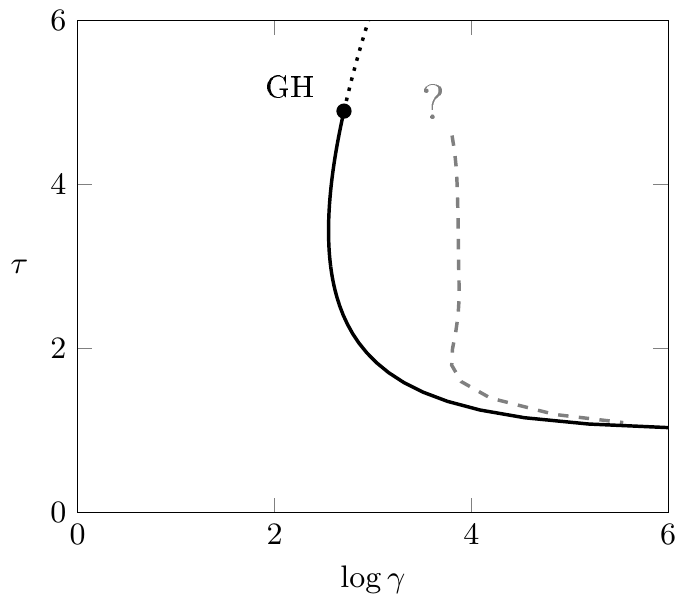}
\caption{Equation \eqref{specialRE}. Left: bifurcation diagram computed with $\tau=3$ and $M=20$ (equilibria and max/min values of the periodic solutions); solid and dashed curves are used to distinguish stable and unstable elements; note the Hopf bifurcation at $\log \gamma \approx 2.5708$ and the period doubling bifurcation at $\log \gamma \approx 3.8777$.
Right: stability curves in the plane $(\log\gamma,\tau)$, supercritical (black solid) and subcritical Hopf (black dotted), and period doubling (gray dashed); note the generalized Hopf point (GH) that characterizes the change in criticality. 
The period doubling curve in the right panel was approximated by performing a sequence of one-parameter continuations with respect to $\log\gamma$, each one for a fixed $\tau$;
we were however not able to complete the curve numerically using MatCont.
}
\label{f:specialRE_bif}
\end{figure}

\begin{figure}[p]
\centering
\includegraphics[width=.95\textwidth]{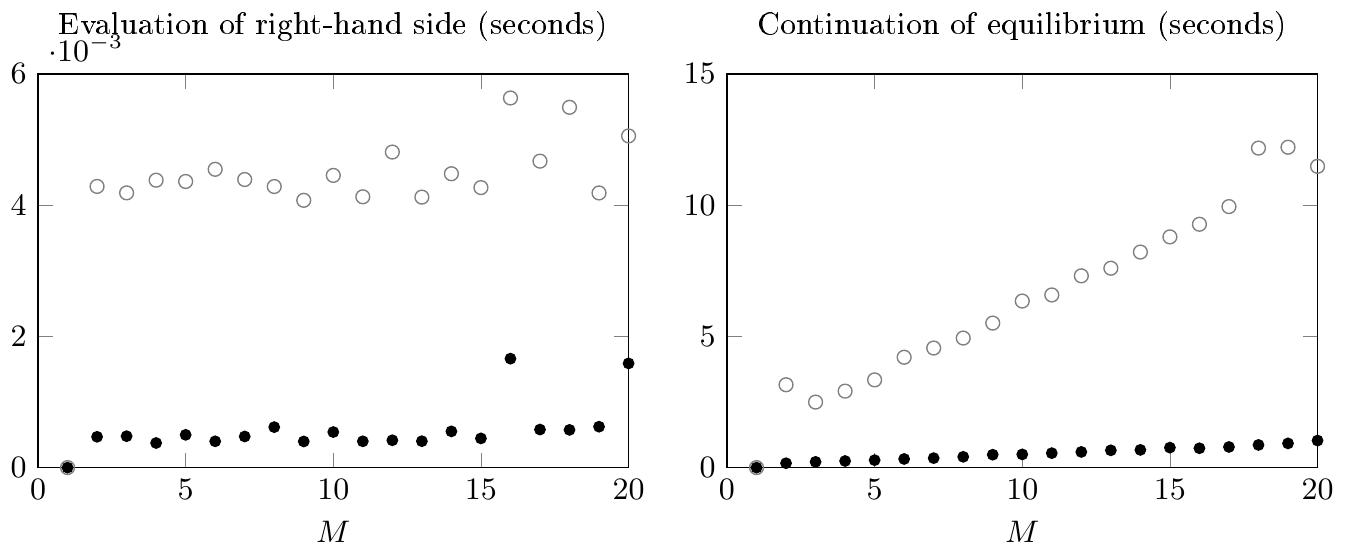} 
\caption{
Equation \eqref{specialRE}. Computation time (seconds) for one evaluation of the right-hand side of the approximating system (left), and for a $50$-point continuation along the branch of positive equilibria (right), varying $M$. 
Comparison between the method in \cite{SIADS2016} (gray $\circ$) and the method proposed here (black $\bullet$).
}
\label{f:specialRE_time}
\end{figure}

\begin{table}[p]
\centering
\caption{Equation \eqref{specialRE}. Computation time for one evaluation of the right-hand side (seconds $\times 10^{-4}$) of the approximating system and for a $50$-point continuation along the branch of positive equilibria (seconds), performed with the \fra{method in \cite{SIADS2016} and with the current method}, and ratio between the two.}
\footnotesize
\begin{tabular}{cccc|cccc}
\toprule
\multicolumn{4}{c|}{\textbf{RHS evaluation}} & \multicolumn{4}{c}{\textbf{Equilibrium continuation}} \\  
$M$ & Method \cite{SIADS2016} & Current method & Ratio 
	& $M$ & Method \cite{SIADS2016} & Current method & Ratio \\
 \midrule
$15$ & 40.14 & 5.65 & 7.10 
	& $15$ & 8.99 & 0.76 & 11.83 \\
$16$ & 43.61 & 6.17 & 7.07 
	& $16$ & 9.52 & 0.73 & 13.04 \\
$17$ & 61.66 & 6.38 & 8.05 
	& $17$ & 10.28 & 0.93 & 11.05 \\
$18$ & 49.17 & 5.13 & 9.58 	
	& $18$ & 10.80 & 0.89 & 12.13 \\
$19$ & 47.53 & 4.41 & 10.71  	
	& $19$ & 11.29 & 0.97 & 11.64 \\
$20$ & 48.77 & 4.75 & 10.27 	
	& $20$ & 11.73 & 1.10 & 10.66 \\
\bottomrule
\end{tabular}
\label{t:times}
\end{table}

Consider the equation
\begin{equation} \label{specialRE}
b(t) = \frac{\gamma}{2} \int_{1}^{\tau} b(t-s) \ee^{-b(t-s)} \diff s, \quad t \geq 0,
\end{equation}
for \fra{$\gamma > 0$} and $\tau>1$, modelling a cannibalism phenomenon in an age-structured population \cite{Breda2013}.
The bifurcation properties of \eqref{specialRE} were recently studied numerically in \cite{EJQTDE2016} using the method \cite{SIADS2016}.

The bifurcation diagram with respect to $\log \gamma$ for fixed $\tau=3$ is plotted in Figure \ref{f:specialRE_bif} (left), obtained with an approximating ODE system of dimension $M=20$.
The nontrivial equilibrium branch undergoes a Hopf bifurcation at $\log \gamma \approx 2.5708$, after which a stable branch of periodic solutions emerges.
The starting of a sequence of period doubling bifurcations is detected on the branch of periodic solutions, with the first period doubling bifurcation at $\log\gamma\approx 3.8777$.
Figure \ref{f:specialRE_bif} (right) shows the bifurcation curves in the plane $(\log\gamma,\tau)$. 
Note that, using a larger discretization index $M=20$, we could approximate the period doubling point with better accuracy than in \cite{EJQTDE2016}.

We used this example for some basic performance comparisons in terms of computational costs.
Figure \ref{f:specialRE_time} shows the computation times of the two approaches (current and \cite{SIADS2016}) for a single evaluation of the right-hand side of the approximating system, and for performing a $50$-point numerical continuation along the branch of positive equilibria using MatCont.
Note the remarkable improvement with the current method, which reduces computation times by approximately a factor 10, see also Table \ref{t:times}.



\section{Nonlinear renewal equations} \label{s:method}

In this section we present the approximation technique for a general scalar \fra{(possibly nonlinear)} RE.
Let $\tau>0$ denote the maximal delay. 
We will work with real-valued functions defined on the domain $[-\tau,0]$, so, for simplicity of notation, we will omit the domain when there is no confusion, writing for instance $L^1$ instead of $L^1([-\tau,0],\mathbb{R})$.

Consider the RE
\begin{equation} \label{RE_b}
b(t) = F(b_t), \qquad t > 0,
\end{equation}
where $F \colon L^1 \to \mathbb{R}$ is a 
\fra{globally Lipschitz continuous function}, and $b_t \in L^1$ denotes the history or state function 
$$ b_t(\theta) := b(t+\theta), \qquad \theta \in [-\tau,0].$$
Equation \eqref{RE_b} is provided with the initial condition
\begin{equation} \label{IC_b}
b(\theta)=\phi(\theta), \qquad \theta \in [-\tau,0],
\end{equation}
for $\phi\in L^1$.
Under these assumptions, the initial value problem \eqref{RE_b}\&\eqref{IC_b} has a unique (global) solution in $[-\tau,+\infty)$,  
which is continuous in $(0, +\infty)$ \cite[Theorem 3.8]{DGG2007}. In view of the analysis of the numerical approach, it is worthwhile 
to remark that $b_t \in C([-\tau,0],\mathbb{R})$ for $t >\tau.$

To efficiently apply the pseudospectral approach and approximately describe the qualitative behavior of the solutions of \eqref{RE_b}, in terms of stability and bifurcations, we want to associate with it an infinite-dimensional dynamical system where the rule for extension is represented explicitly, acting in a state space where point evaluation is well defined.

For this purpose we define
\begin{equation}\label{B}
B(t) := \int_0^t b(s) \diff s, \qquad t \geq -\tau,
\end{equation}
and, \fra{for all $t\geq 0$,}
\begin{equation} \label{v}
v(t)(\theta) := B(t+\theta)-B(t),\qquad \theta \in [-\tau,0],
\end{equation}
so that we get for all $t \geq 0$
\begin{equation} \label{dt}
\frac{\diff}{\diff t} v(t)(\theta) = b_t(\theta)-b_t(0)
\end{equation}
and
\begin{equation} \label{dtheta}
\frac{\diff}{\diff \theta} v(t) = b_t.
\end{equation}
\fra{We stress that \eqref{dtheta} should be interpreted as follows: presuming that $v(t)$ is an absolutely continuous function of $\theta$, it is almost everywhere differentiable and the derivative is Lebesgue integrable, so it defines an element of $L^1$, i.e., an equivalence class, denoted by $b_t$.
}

To derive an abstract differential equation from \eqref{dt}\&\eqref{dtheta}, inspired by the nonlinear theory in \cite{DGG2007} and by the theoretical framework in \cite{TwinSemigroups}, we introduce the space \fra{$NBV = NBV([-\tau,0],\mathbb{R})$ defined as}
\begin{align*}
NBV := \{ &\psi\in BV \fra{([-\tau,0],\mathbb{R})} \colon \psi(0)=0 \\
	&\text{ and $\psi$ is continuous from the right on the open interval } (-\tau,0) \}
\end{align*}
and we \fra{represent $L^1$ by} the subspace $AC_0$ of the absolutely continuous functions in $NBV$, i.e.,
$$ \fra{AC_0} = \{ \psi \in AC \colon \psi(0)=0\},$$
which is a Banach space with norm
$$ \|\psi\|_{NBV} = \|\psi'\|_{L^1} = \int_{-\tau}^0 |\psi'(\theta)| \diff \theta, \qquad \psi \in \fra{AC_0}.$$
The operator
\begin{equation*} 
(V \phi)(\theta) = -\int_\theta^0 \phi(s) \diff s, \qquad \phi \in L^1,\ \theta \in [-\tau,0],
\end{equation*}
maps an element of $L^1$ (i.e., an equivalence class) to a function in \fra{$AC_0$}, defining an embedding of $L^1$ into $NBV$. 
Although we defined $V$ with domain in $L^1$, in the following we will also consider the restrictions $V\big|_{NBV}$ and $V\big|_C$ to the spaces $NBV$ and $C$. We will sometimes omit the explicit restriction and use simply the notation $V$, since the appropriate domain should be clear from the context.
\fra{We can reformulate \eqref{B} and \eqref{v} as}
\begin{equation} \label{Vbt}
v(t) = V b_t.
\end{equation}
Note that for the constant function $\bar{b}$ we get \fra{$v(t)=\bar{v}$ with} $\bar{v}(\theta)= V \bar{b}(\theta) =\theta \bar{b}$, $\theta \in [-\tau,0],$ while $v(t) = V b_t$ is periodic when $b(t)$ is a periodic function.

\medskip
We now define the operator $C_0 \colon D(C_0) (\subset NBV) \to NBV$ as $C_0:= (V\big|_{NBV})^{-1}$. 
That is, we define
$$ D(C_0) = 
\{ \psi \in \fra{AC_0} \colon \psi = V\phi \text{ for some } \phi \in NBV \}
$$
and, for $\psi \in D(C_0)$, we take 
$$C_0 \psi = \{ \phi \colon \psi = V \phi\}.$$
The density of \fra{(the embedding of) $NBV$ in $L^1$ (defined by the identity, i.e., by associating to any $NBV$ element the equivalence class of functions that are almost everywhere equal to it)\footnote{For a proof of the density of (the embedding of) $NBV$ in $L^1$, consider $f \in L^1([-\tau,0])$, extended by 0 outside the interval $[-\tau,0]$. Let $g$ be defined by $g(\theta):= -\int_{\theta}^0 f(\sigma)\diff \sigma$, and, for $t>0$, let $h_t$ be defined by $h_t(\theta):= \frac{1}{t}[g(t+\theta)-g(\theta)]$. Then $g,h_t \in AC_0 \subset NBV$. From \cite[Appendix II, Theorem 2.3]{DiekmannBook}, we have $\| h_t - g'\|_{L^1} \to 0$ as $t\to 0^+$, with $g'=f$.}  
}
implies that $\fra{AC_0}=\overline{D(C_0)}$.
Note that $C_0$ is multivalued, since functions that differ only by a jump in $\theta=-\tau$ are mapped by $V$ to the same element of $NBV$. 
We could eliminate this ambiguity by considering the restriction of $V$ to $NBV_0$, defined as the subspace of $NBV$ consisting of functions that are continuous in $-\tau$. In this way, the inverse $(V\big|_{NBV_0})^{-1}$ is single-valued. Note that $AC_0 \subset NBV_0$, and $NBV_0$ naturally arises in the sun-star approach, as explained in \cite{DGG2007}. This restriction is not necessary for the numerical approach proposed in this manuscript, so we will keep working with the larger space $NBV$. 

\fra{For $v(t) \in D(C_0)$, the inverse relation \eqref{dtheta} can now be written as $b_t =$ ``$C_0 v(t)$'', where ``$\cdot$'' indicates that we take the equivalence class in $L^1$, to which all elements in $C_0 v(t)$ belong; however, the request that $v(t)$ belongs to $D(C_0)$ is stronger than needed, as \eqref{dtheta} is well defined also when $v(t) \in AC_0$.}

When we focus on restrictions to $NBV_0$, $C_0$ is the generator (in weak$^*$ sense) \cite{DiekmannBook} of the semigroup $\{S_0(t)\}_{t \geq 0}$ defined on $NBV$ by
$$ (S_0(t)\psi)(\theta) =
\begin{cases}
\psi(t+\theta), & t+\theta \leq 0, \\
0, & t+\theta>0.
\end{cases}
$$
The semigroup $\{S_0(t)\}_{t\geq 0}$ is not strongly continuous on $NBV$, but it is strongly continuous on the subspace $AC_0$.

To interpret the rule for extension \eqref{RE_b} as perturbation of the translation semigroup $\{S_0(t)\}_{t\geq 0}$, in analogy with the nonlinear theory treated in \cite{DGG2007} and the linear perturbation theory developed in \cite{TwinSemigroups}, we define $q \in NBV$ by
\begin{equation} \label{q}
 q(\theta) = \begin{cases}
0, & \theta =0, \\
-1, & \theta \in [-\tau,0),
\end{cases}
\end{equation}
i.e., $q$ is the Heaviside function that represents the Dirac measure in $\theta=0$. Note that $q \notin AC_0$ because of the discontinuity in $0.$

\medskip
For $\psi:=V\phi \in AC_0$, \eqref{dt}\&\eqref{dtheta}\&\eqref{RE_b}
is formally equivalent to the abstract Cauchy problem for $v(t) \in AC_0$
\begin{align}
\frac{\diff v(t)}{\diff t} &= C_0 v(t) + q F(C_0 v(t)), \qquad t\geq 0,  \label{ADE} \\
v(0) &= \psi, \label{IC_ADE}
\end{align}
in a sense that is made precise, for the linear case, in \cite{TwinSemigroups}.
In particular, given a (classical) solution $v$ of \eqref{ADE}\&\eqref{IC_ADE}, we reconstruct the solution of \eqref{RE_b}\&\eqref{IC_b} by defining 
$b(t)=F(C_0 v(t))$ for $t>0$, \fra{and \eqref{Vbt} holds}. 

It is then clear that the linear operator $C_0$ captures the translation, while the rule for extension is represented explicitly by the nonlinear perturbation $q F \circ C_0$ (where $\circ$ denotes composition).

\subsection{Reduction to ODE via pseudospectral discretization}

To approximate the solution $v(t)$ of \eqref{ADE}\&\eqref{IC_ADE}, we fix a \emph{discretization index} $M\in \mathbb{N}$ and consider a set of points $\Theta_M =\{\theta_1,\dots,\theta_M\} \subset [-\tau,0)$ and $\theta_0=0$, with
$$-\tau \leq \theta_M < \cdots < \theta_1 < \theta_0=0.$$ 
In the numerical simulations we consider the \emph{Chebyshev zeros}, i.e., the roots of the Chebyshev polynomial of the first kind of degree $M$ \cite{Xu2016}, transformed to the interval $(-\tau,0)$, which are given explicitly by
\begin{equation}\label{cheb_zeros}
\theta_j = \frac{\tau}{2} \left( \cos \left( \frac{2j-1}{2M}\pi \right) - 1 \right), \qquad j=1,\dots,M.
\end{equation}
Let $\ell_j$, $j=0,\dots,M$, be the Lagrange polynomials associated with $\{\theta_0=0\} \cup \Theta_M$, which are defined by
$$ \ell_j(\theta) := \prod_{\substack{k=0\\ k\neq j}}^M \frac{\theta-\theta_k}{\theta_j - \theta_k}, \qquad \theta \in [-\tau,0].$$
We recall that $\ell_j(\theta_k)=\delta_{jk}$, where $\delta_{jk}$ denotes the Kronecker symbol, and that, for all $\theta \in [-\tau,0]$, the Lagrange polynomials have the properties
\begin{align*}
\sum_{j=0}^M \ell_j(\theta) =1, \qquad 
\sum_{j=0}^M \ell_j'(\theta)=0,
\end{align*}
and form a basis of the space of polynomials of degree less than or equal to $M$.
Using the Lagrange basis, the unique $M$-degree polynomial $p$ interpolating a function $\varphi$ with well-defined point values on the nodes $\{0\} \cup \Theta_M$ can be expressed by
$$ p(\theta) = \sum_{j=0}^M \varphi(\theta_j) \ell_j(\theta), \qquad \theta\in [-\tau,0].$$

To obtain a finite dimensional approximation of \eqref{ADE} we need to approximate the operator $C_0$ and the function $F$.
To do this, we introduce $P_M \colon \mathbb{R}^M \to NBV$, the interpolation operator on $\{0\} \cup \Theta_M$ with value zero in $\theta_0=0$, and $R_M \colon NBV \to \mathbb{R}^M$, the restriction on $\Theta_M$. More precisely,
\begin{align}
P_M y &:=  \sum_{j=1}^M y_j \ell_j, \qquad y \in \mathbb{R}^M, \label{defP}\\
(R_M \varphi)_j &:= \varphi(\theta_j), \qquad \varphi \in NBV,\ j=1,\dots,M. \notag 
\end{align}
The operator $C_0$ is approximated by pseudospectral differentiation, i.e., by the finite dimensional operator $D_M \colon \mathbb{R}^M \to \mathbb{R}^M$ defined by 
$$ D_M := R_M C_0 P_M.$$
At this point, the choice of a mesh like \eqref{cheb_zeros}, which does not include the point $-\tau$, is even better justified, as the multi-valuedness of $C_0$ disappears in the pseudospectral approximation. If the point $-\tau$ belongs to $\Theta_M$, then we should define $C_0$ as the (single-valued) inverse of the restriction of $V$ on $NBV_0$, so that $D_M$ is well defined. 

To write explicitly the operator $D_M$, we consider a vector $y=(y_1,\dots,y_M) \in \mathbb{R}^M$. Then $P_M y$ is given by \eqref{defP} and, for all $k=1,\dots,M$,
\begin{align*}
(D_M y)_k &= (C_0 P_M y)(\theta_k) \\
	&= \sum_{j=1}^M y_j \ell_j'(\theta_k).
\end{align*}
The entries of the matrix $D_M$ are therefore given explicitly by
$$ (D_M)_{kj} := \ell_j'(\theta_k), \qquad k,j=1,\dots,M.$$
Note that the elements $(D_M)_{kj}$ are entries of the differentiation matrix associated with the mesh $\{0\} \cup \Theta_M$ \cite{Trefethen2000}.
More precisely, $D_M$ is the submatrix of the differentiation matrix accounting for the fact that functions are zero in zero.

\medskip
To approximate the perturbation $qF\circ C_0$, which captures the rule for extension, we let $\mathbf{1} \in \mathbb{R}^M$ denote the vector with all the entries equal to $1$ \fra{(note that the vector $-\mathbf{1}$ is the discretization of $q$)}, and introduce $F_M \colon \mathbb{R}^M \to \mathbb{R}$ by
$$ F_M := F \circ C_0 P_M.$$
Using \eqref{defP}, for $y \in \mathbb{R}^M$ we have
$$ F_M(y) = F(\sum_{j=1}^M  \ell_j' y_j).$$
Then, $qF\circ C_0$ is approximated by $-\mathbf{1} F_M$.

\medskip
Putting everything together, we obtain the following finite dimensional ODE system
\begin{equation} \label{ODEM}
\frac{\diff x}{\diff t} = D_M x - F_M (x) \mathbf{1},
\end{equation}
for $x(t) \in \mathbb{R}^M$, $t \geq 0.$ By construction, the $(M-1)$-degree polynomial $C_0 P_Mx(t)= \sum_{j=1}^M \ell_j' x_j (t)$ furnishes an approximation of $b_t$, while $F_M (x(t))$ approximates $b(t)$ in \eqref{RE_b}. The computational advantage of \eqref{ODEM} with respect to the ODE system derived in \cite{SIADS2016} is evident: the nonlinear contribution appears explicitly as a perturbation of the linear system, and there is no need to solve a nonlinear equation to impose the domain condition. 

The dynamics of \eqref{ODEM} can be investigated efficiently by using available software for numerical continuation and bifurcation study. It is therefore important to understand to which extent \eqref{ODEM} mimics the dynamics of the original infinite dimensional dynamical system described by the RE \eqref{RE_b}. Here we examine the equilibria and their stability properties.

\subsection{Correspondence of equilibria and linearized equations}
In what follows we assume that the operator $F$ in \eqref{RE_b} is continuously Fr\'echet differentiable and we focus on constant solutions and the corresponding linearized equations.
\fra{We slightly abuse notation by using the symbol $\overline b$ to represent both a real number and the constant function (defined on $[-\tau,0]$) taking that number as its one and only value.}

\begin{theorem}
The equilibria of \eqref{RE_b} and \eqref{ODEM} are in one-to-one correspondence. 
More precisely, if $\overline{b}=F(\overline{b})$, then $\overline{x} \in \mathbb{R}^M$ with
\begin{equation}\label{equil}
\overline{x}_j = \overline b \, \theta_j, \qquad j=1,\dots,M,
\end{equation}
satisfies $D_M \overline{x}- F_M(\overline{x}) \mathbf{1} = 0$ and so defines an equilibrium of \eqref{ODEM};
vice versa, if $D_M \overline{x} - F_M(\overline{x}) \mathbf{1} = 0$ then $\overline{b} = \overline{x}_j/\theta_j$ does not depend on $j$ and satisfies $\overline{b}=F(\overline{b})$.
\end{theorem}

\begin{proof}
We first derive a useful identity.
Since $M$-degree polynomial interpolation is exact on all polynomials with degree equal to or less than $M$, for the identity function we can write
$$ \theta = \sum_{j=1}^M \theta_j \ell_j(\theta), \qquad \theta \in [-\tau,0].
$$
By differentiation we then obtain
\begin{equation} \label{1_sum_der}
1 = \sum_{j=1}^M \theta_j \ell_j'. 
\end{equation}

Assume first $\overline{b}=F(\overline b)$, and define $\overline x$ by \eqref{equil}. From \eqref{1_sum_der},
$$ 
\overline b = \overline b \sum_{j=1}^M \theta_j \ell_j'(\theta)$$
for all $\theta \in [-\tau,0]$ and, \fra{consequently,} 
\fra{$$[D_M \overline x]_k = \sum_{j=1}^M \ell_j'(\theta_k) \overline x_j = \overline b \sum_{j=1}^M \theta_j \ell_j'(\theta_k) = \overline b,$$
i.e., }
$D_M \overline x = \overline b \, \mathbf{1}$.
Therefore we conclude
$$ D_M \overline{x} - F_M(\overline{x})\mathbf{1} = D_M \overline x - F(\overline b \sum_{j=1}^M \theta_j \ell_j')\, \mathbf{1} = \overline b \, \mathbf{1} - F(\overline b)\, \mathbf{1} =0,$$
\fra{or, in words,} $\overline x$ is an equilibrium of \eqref{ODEM}.

Vice versa, assume that $\overline x \in \mathbb{R}^M$ satisfies
\begin{equation} \label{eq_ode}
 D_M \overline x - F_M(\overline{x}) \,\mathbf{1} = 0,
\end{equation}
and define $\overline b := F_M(\overline{x}) = F(\overline{p}_{M-1})$, with $\overline{p}_{M-1}:= \sum_{j=1}^M \overline x_j \ell_j'$. Note that $\overline{p}_{M-1}$ is a polynomial with degree $M-1$ and, by definition, $(D_M \overline x)_k = (\sum_{j=1}^M \overline x_j \ell_j')(\theta_k) = \overline{p}_{M-1}(\theta_k)$, $k=1,\dots,M$.
By \eqref{eq_ode}, we conclude that $\overline{p}_{M-1}$ takes the value $\overline{b}$ on $M$ distinct points and therefore $\overline{p}_{M-1}$ is identically equal to $\overline b$. It now follows from the definition of $\overline{b}$ that $\overline{b}=F(\overline{b})$.
\end{proof}

\begin{theorem}
The operations of linearization around an equilibrium and pseudospectral discretization commute.
\end{theorem}
\begin{proof}
Let $\overline b$ be an equilibrium of \eqref{RE_b}. The linearization around $\overline b$ reads
$$ b(t) = DF(\overline b) \, b_t,$$
and the corresponding approximating ODE system is
\begin{align}
\frac{\diff x}{\diff t} &= D_M x - \big[ DF(\overline b) \sum_{j=1}^M x_j \ell'_j \big] \, \mathbf{1}. \label{lin_discr}
\end{align}

On the other hand, let $\overline x$ be an equilibrium of \eqref{ODEM} with \eqref{equil}.
The linearization of \eqref{ODEM} around $\overline x$ reads
\begin{align*}
\frac{\diff x}{\diff t} &= D_M x - [DF_M(\overline{x}) x] \mathbf{1} \\
	&= D_M x - \Big[ DF(\overline b \, \sum_{j=1}^M \theta_j \ell_j') \sum_{j=1}^M x_j \ell_j'\Big] \, \mathbf{1} \\
	&=D_M x - \Big[ DF(\overline b)  \sum_{j=1}^M x_j \ell_j'\Big] \,\mathbf{1},
\end{align*}
where in the last step we used \eqref{1_sum_der}. Hence the linearized system coincides with \eqref{lin_discr}.
\end{proof}

The principle of linearized stability \cite{DGG2007} ensures that the local stability properties of an equilibrium $\overline{b}$ of \eqref{RE_b} are determined by the stability properties of the zero solution of the system linearized around $\overline{b}$, if the equilibrium is hyperbolic.
Therefore the next step in order to study whether the finite dimensional ODE \eqref{ODEM} approximates the stability properties of the equilibria of \eqref{RE_b} is to focus on linear equations.

\section{The linear case: convergence analyses} \label{s:linear_conv}

In this section we focus on the linear case and study the convergence of exponential solutions, which are related to stability.
We consider the linear RE
\begin{equation} \label{linear_RE_b}
b(t) = \int_0^\tau k(a) b(t-a) \diff a, \qquad t > 0,
\end{equation}
where $k$ is a bounded measurable function on $[0,\infty)$ with support in $[0,\tau]$. 

\medskip
Following \cite{TwinSemigroups}, let $\langle \cdot , \cdot \rangle$ denote the pairing between bounded measurable functions and $NBV$, i.e.,
$$ \langle k, \psi \rangle := \int_{-\tau}^0 k(-a) \psi(\diff a), \qquad \psi \in NBV.$$
In particular, if $\psi \in \fra{AC_0}$, the pairing reads
$$ \langle k, \psi \rangle = \int_{-\tau}^0 k(-a) \psi'(a) \diff a.$$
By defining $K_M$ as the row vector with components
\begin{equation} \label{KM}
(K_M)_j := \langle k, \ell_j \rangle = \int_{-\tau}^0 k(-a) \ell_j'(a) \diff a, \qquad j=1,\dots,M,
\end{equation}
we can write the ODE system \eqref{ODEM} approximating \eqref{linear_RE_b} as
\begin{equation} \label{linear_ODEM}
\frac{\diff x}{\diff t} = D_M x - (K_M x) \, \mathbf{1}.
\end{equation}

We want to study in which sense the finite dimensional approximation \eqref{linear_ODEM} captures the stability properties of the zero equilibrium of \eqref{linear_RE_b} as $M\to \infty$. 
The asymptotic stability of the zero equilibrium is determined by the roots of the characteristic equation, a.k.a.\ ``characteristic roots''.
Specifically, \eqref{linear_RE_b} has an exponential solution of the form $b(t)= \alpha \ee^{\lambda t}$ with $\alpha \neq 0$ if and only if $\lambda \in \mathbb{C}$ is a root of the characteristic equation
\begin{equation}  \label{CE}
1 = \int_0^\tau k(a) \ee^{-\lambda a} \diff a.
\end{equation}
The real parts of the characteristic roots $\lambda$ determine the stability of the zero solution of \eqref{linear_RE_b}: if $\Re\lambda <0$ for all $\lambda$, then the zero equilibrium is asymptotically stable; if there exists a root $\lambda$ with $\Re\lambda>0$, then the zero equilibrium is unstable.

\medskip
From now on we assume that the set $\Theta_M$ is chosen such that the associated Lebesgue constant, defined as 
$$\tilde\Lambda_M := \max_{\theta \in [-\tau,0]} \sum_{j=1}^M |\tilde \ell_j(\theta)|,$$
for $\tilde \ell_j$ Lagrange polynomials associated with $\Theta_M$, satisfies
\begin{equation} \label{cond_Lebesgue}
\lim_{M\to \infty} \frac{\tilde \Lambda_M}{M} = 0.
\end{equation} 
We remark that \eqref{cond_Lebesgue} is true for instance for the nodes \eqref{cheb_zeros}, for which $\tilde \Lambda_M = O(\log M)$ \cite[Chapter 1.4.6]{MastroianniBook}.
We also remark that including additional nodes, for instance the extremum $-\tau$, does not affect the validity of \eqref{cond_Lebesgue}, \cite[Chapter 4.2]{MastroianniBook}.

The other important ingredient for convergence is the regularity of the function $\varphi$: the smoother a function, the faster the convergence of its interpolants.
More specifically, let $\LL_{M-1}$ denote the polynomial interpolation operator on $\Theta_M$, such that $\LL_{M-1} \varphi$ is the unique $(M-1)$-degree polynomial with $[\LL_{M-1}\varphi] (\theta_j) = \varphi(\theta_j)$, $j=1,\dots,M$.
Although $\LL_{M-1}$ can be defined also for bounded variation functions, and some results for interpolation in $L^p$ norm exist \cite{Mastroianni2010}, we here require that $\varphi$ is at least continuous, and use error bounds of polynomial interpolation in terms of the uniform norm.
In the following, let $c$ denote a constant (different from time to time) independent of $M$. 
If $\varphi$ is Lipschitz continuous, the bound
\begin{equation} \label{interp_error_sup}
 \| (I - \LL_{M-1})\varphi \|_{\infty} \leq c \, \frac{\tilde \Lambda_M}{M} \, \text{Lip}(\varphi) 
 \end{equation}
holds, where Lip$(\varphi)$ denotes the Lipschitz constant of $\varphi$, \fra{and, under the assumption \eqref{cond_Lebesgue}, the right-hand side of \eqref{interp_error_sup} tends to zero as $M \to \infty$.}
The order of convergence depends on the regularity of the interpolated function. More precisely, the estimate
$$ \| (I - \LL_{M-1})\varphi \|_{\infty} \leq c \, \frac{\tilde \Lambda_M}{M^k} \, \|\varphi^{(k)}\|_\infty $$
holds for $\varphi \in C^k$, 
see for instance \cite[Chapters 1.2 and 1.4]{MastroianniBook}.
For analytic functions, the uniform error of polynomial interpolation behaves as $O(\rho^{-M})$ for some $\rho>1$, a phenomenon called \emph{spectral accuracy}, or \emph{geometric convergence}, or \emph{exponential convergence} \cite{MastroianniBook,Trefethen2013,Trefethen2000}.
Finally we remark that the uniform convergence of polynomial interpolation holds also for absolutely continuous functions interpolated at Chebyshev zeros \cite{Krylov}. 

Before studying the approximation error, we note that the numerical computation of the vector $K_M$ in \eqref{KM} may involve the use of quadrature formulas to approximate the integrals. These approximations can introduce numerical errors in addition to those due to the pseudospectral approximation. In the following, we ignore these errors by assuming that the integrals are computed exactly.

\subsection{Resolvent operators}

Before focusing on characteristic roots, we study more in general the convergence of the pseudospectral approximations of the resolvent operators of $C_0$. Although the convergence of resolvent operators is not directly used to prove the convergence of the stability criteria of equilibria, it is fundamental to set the basis for a broader study of the approximation error, as solution operators are strictly connected with the resolvent of their generator via the Laplace transform.
For every $\lambda \in \mathbb{C}$ and $\varphi \in NBV$, the resolvent operator of $C_0$ is given by
\begin{equation} \label{resolvent}
((\lambda I-C_0)^{-1}\varphi) (\theta) = \ee^{\lambda \theta} \int_{\theta}^0 \ee^{-\lambda s} \varphi(s) \diff s, \quad \theta \in [-\tau,0].
\end{equation}
\fra{In the following analysis of the approximation error, we require that $\varphi \in NBV$ is continuous.}

\begin{lemma} \label{l:res}
\fra{Let $\varphi \in NBV \cap C$, let $B$ be a bounded open subset of $\mathbb{C}$, and assume that \eqref{cond_Lebesgue} holds. 
There exists $\overline{M}(B)$ such that, for any index $M \geq \overline{M}(B)$ and $\lambda \in B$, the polynomial 
\begin{equation}
 \label{pMcoll} 
 p_M :=P_M(\lambda I - D_M)^{-1} R_M \varphi
\end{equation}
is well defined, and, for $\psi := (\lambda I - C_0)^{-1} \varphi,$
\begin{equation} \label{psi-p}
\| \psi - p_M \|_{NBV} \leq c(B) \; \big\| r_M(\lambda,\varphi) \big\|_{\infty},
\end{equation}
where $ r_M(\lambda,\varphi) := (I-\LL_{M-1}) \psi'$ and 
$$ c(B) = 2 \tau \, \sup_{\lambda \in \overline{B}} \left( 1 + |\lambda| \frac{1-\ee^{-\tau \Re \lambda}}{\Re \lambda} \right).
$$
%
%
%
}
\end{lemma}

\begin{proof}
The proof technique is similar to the one used in \cite[Proposition 5.1]{BMVBook} and \cite[Theorem 4.1]{DCDS2020Vermiglio}. We therefore only sketch the main steps.

The function $\psi$ satisfies 
\begin{equation} \label{psi}
\begin{cases}
\psi'(\theta) = \lambda \psi(\theta) - \varphi(\theta), & \theta \in [-\tau,0], \\
\psi(0)=0.
\end{cases}
\end{equation}
\fra{For given $\varphi \in NBV$, the function $\psi$ defined by \eqref{resolvent} satisfies \eqref{psi}. Here, since $\varphi$ is continuous, we have $\psi' \in C$ and the first identity also holds in $\theta=0$.}

On the other hand, proving that \eqref{pMcoll} is well defined is equivalent to showing that the collocation problem 
\begin{equation*} 
\begin{cases}
p_M'(\theta) = \lambda p_M(\theta) - \varphi(\theta), & \theta=\theta_1,\dots,\theta_M, \\
p_M(0)=0
\end{cases}
\end{equation*}
admits a unique solution.
Since $p_M'$ has degree $M-1$, it can be expressed as interpolation polynomial on $\Theta_M$, i.e., as
\begin{equation} \label{pM'_gen}
p_M' 
= \LL_{M-1} (\lambda p_M - \varphi).
\end{equation}
Define $e_M:= \psi-p_M$. Subtracting \eqref{pM'_gen} from the first equation in \eqref{psi} and writing $e_M = V e_M'$, we obtain
\begin{equation} \label{eq_error}
e_M' = \lambda \LL_{M-1} V e_M' + \lambda (I-\LL_{M-1}) \psi - (I-\LL_{M-1})\varphi.
\end{equation}
Note that, since $\varphi, \psi,\psi' \in C$, \eqref{eq_error} can be interpreted as an equation in $C$. We show that, for all $\lambda \in \mathbb{C}$, the operator $(I-\lambda \LL_{M-1} V)$ is invertible in $C$ by showing that it is a perturbation, in supremum norm, of (the restriction to $C$ of) the operator $(I-\lambda V)$. In fact, since $V\varphi$ is Lipschitz for all $\varphi \in C$, from \eqref{interp_error_sup} we can bound
$$ \| \lambda (I - \LL_{M-1}) V \|_\infty \leq c \, \frac{\tilde \Lambda_M}{M} |\lambda|,$$
with $c$ independent of $M$.
In $C$, $(I-\lambda V)$ is invertible with
\begin{equation*} 
[(I-\lambda V)^{-1}\zeta](\theta) = \zeta(\theta) - \lambda \int_\theta^0 \ee^{\lambda (\theta-s)} \zeta(s) \diff s, \qquad \zeta\in C,
\end{equation*}
and
$$ \fra{
\|(I-\lambda V)^{-1} \|_{\infty} \leq 
\begin{cases}
1 + |\lambda| \frac{1-\ee^{-\tau \Re \lambda}}{\Re \lambda}, & \text{if } \Re \lambda \neq 0 \\[6pt]
1 + \tau \, |\lambda | , & \text{if } \Re \lambda = 0.
\end{cases}
}
$$
\fra{For $\Re \lambda \neq 0$}, from Banach's perturbation lemma (e.g., \cite[Theorem 10.1]{KressBook}), by taking $\tilde{M}(\lambda)$ such that $\|\lambda (I-\LL_{M-1})V\|_\infty \, \|(I-\lambda V)^{-1} \|_{\infty} < 1/2$ for all $M \geq \tilde{M}(\lambda)$, we conclude that $(I - \lambda \LL_{M-1}V)$ is invertible in $C$ for all $M \geq \tilde{M}(\lambda)$, and
\fra{
\begin{equation} \label{bound_discr} 
\| (I - \lambda \LL_{M-1} V)^{-1}\|_{\infty} \leq 
\begin{cases}
2 \left(1 + |\lambda| \frac{1-\ee^{-\tau \Re \lambda}}{\Re \lambda}\right), & \text{if } \Re \lambda \neq 0 \\[6pt]
2 (1+\tau |\lambda |), & \text{if } \Re \lambda = 0.
\end{cases}
\end{equation}
}
Hence, from \eqref{eq_error}, 
\begin{equation*} 
e_M' = (I - \lambda \LL_{M-1} V)^{-1} r_M(\lambda,\varphi),
\end{equation*}
with 
\begin{align*}
r_M(\lambda,\varphi) := \lambda (I-\LL_{M-1}) \psi - (I-\LL_{M-1})\varphi 
	= (I-\LL_{M-1}) \psi'.
\end{align*}
The bound \eqref{psi-p} follows from the fact that $ e_M = V e_M' \in \fra{AC_0}$ and therefore
$$ \| e_M \|_{NBV} \leq \tau \| e_M'\|_{\infty},$$
\fra{taking $\overline{M}(B) = \sup_{\lambda \in \overline{B}} \tilde{M}(\lambda)$ and using the fact that the right-hand side of \eqref{bound_discr} is continuous on $\mathbb{C}$.}
\end{proof}

From \eqref{psi-p}, it is clear that $\|\psi_M-p_M\|_{NBV} \to 0$ when $\|r_M\|_\infty \to 0$.
\fra{If $\varphi \in AC_0$}, $\psi' \in \fra{AC_0}$ and therefore $\|r_M\|_\infty \to 0$ when the grid consists of the Chebyshev zeros \eqref{cheb_zeros}, cf.~\cite{Krylov}.
If $\varphi \in D(C_0)$, then $\psi'' \in NBV$ and the convergence $\|r_M\|_\infty \to 0$ holds for any set of nodes satisfying \eqref{cond_Lebesgue}, since
$$ \|(I-\LL_{M-1})\varphi\|_\infty \leq \tilde \Lambda_M \mathcal{E}_{M-1}(\varphi) \leq \frac{2 \tilde \Lambda_M }{\pi (M-2)} \| \psi'' \|_{NBV}, $$
where $\mathcal{E}_M$ denotes the error of the best uniform approximation in $C$ with polynomials of degree $M$ \cite[Chapter 7]{Trefethen2013}.
As already discussed, the order of convergence $\|r_M\|_{\infty} \to 0$ is higher if $\psi'$ has higher regularity (at least Lipschitz continuous, see \eqref{interp_error_sup} and following discussion).

\medskip
\fra{The proof of Lemma \ref{l:res} can be extended to functions $\varphi \in NBV$ that are continuous in $[-\tau,0)$, but have a jump discontinuity in $\theta=0$ (note that this class includes the function $q$).
In fact, note that $(\lambda I-C_0)^{-1}$ maps $NBV$ to $D(C_0)$, hence $\psi = (\lambda I-C_0)^{-1}\varphi \in \fra{AC_0}$.
From \eqref{psi}, we note that if $\varphi$ is discontinuous in $\theta=0$, 
$\psi'$ is discontinuous, too. 
However, we can adapt the previous proof by considering $\varphi$ and $\psi'$ extended by continuity in $\theta=0$, hence working in the space $C$. 
In particular, the following special case for the function $q\in NBV$ is relevant for the analysis of characteristic roots. 
}

\begin{lemma} \label{l:res_1}
\fra{
Let $B$ be an open bounded subset of $\mathbb{C}$ and assume that \eqref{cond_Lebesgue} hold. 
There exists $\overline M (B)$ such that, for any index $M \geq \overline M(B)$ and $\lambda \in B$, the polynomial
\begin{equation*} 
p_{M,\lambda}:= P_M(\lambda I - D_M)^{-1} (-\mathbf{1})
\end{equation*}
is well defined, and, for $\psi_\lambda := (\lambda I - C_0)^{-1}q$,
\begin{equation} \label{error_m_q}
\| \psi_\lambda - p_{M,\lambda} \|_{NBV} \leq C_2(B) \frac{C_1(B)^M}{M!}
\end{equation}
where
\begin{align*}
C_1(B) &= \tau \, \max_{\lambda \in \overline{B}} |\lambda|, \\
C_2(B) &= 2 \, \sup_{\lambda \in \overline{B}} \left[ \left( 1 + |\lambda| \frac{1-\ee^{-\tau \Re \lambda} }{\Re \lambda} \right) \,  \max \{ \ee^{-\tau \Re \lambda}, 1\} \right] . 
\end{align*}
}
\end{lemma}

\begin{proof}
\fra{From \eqref{resolvent} we obtain, for $\theta \in [-\tau,0]$,
\begin{equation} \label{psi_lambda}
\psi_{\lambda}(\theta)= 
\begin{cases}
\displaystyle{\frac{\ee^{\lambda \theta}-1}{\lambda}}, & \text{if } \lambda \neq 0, \\
\theta, & \text{if } \lambda = 0.
\end{cases} 
\end{equation}
Consider now the function $\eta_\lambda(\theta) := \ee^{\lambda \theta}$, $\theta \in [-\tau,0]$, which satisfies
\begin{equation*}
\eta_\lambda(\theta) = \lambda \psi_\lambda(\theta) + q(\theta), \qquad \theta \in [-\tau,0),
\end{equation*}
(i.e., $\eta_\lambda = \psi_\lambda'$) and $\eta_\lambda$ is continuous on $[-\tau,0]$.
}

\fra{As in the proof of Lemma \ref{l:res}, let $p_{M,\lambda}$ be the solution of 
\begin{equation*}
\begin{cases}
p_{M,\lambda}'(\theta) = \lambda p_{M,\lambda}(\theta) + q(\theta), & \theta = \theta_1,\dots,\theta_M, \\
p_{M,\lambda}(0)=0,
\end{cases}
\end{equation*}
and let $z_M := \eta_\lambda - p_{M,\lambda}'$. Then $e_M := \psi_\lambda -p_{M,\lambda} = V z_M$, and $z_M$ satisfies
\begin{equation} \label{eq_error_2}
z_M = \lambda \LL_{M-1} V z_M + \overline r_M(\lambda),
\end{equation}
for $\overline r_M(\lambda):= (I-\LL_{M-1}) \eta_\lambda$. 
Since $\overline r_M(\lambda) \in C$, we can invert equation \eqref{eq_error_2} using \eqref{bound_discr}, 
and the assertion follows from $e_M = V z_M$.
}

\fra{
Using the Cauchy interpolation remainder \cite[Theorem 3.1.1]{Davis1975}, we can bound
\begin{equation} \label{rm}
\| \overline r_M(\lambda) \|_\infty \leq \frac{\tau^M \| \psi^{(M)}\|_\infty}{M!} \leq \frac{\tau^M |\lambda|^M \max \{ \ee^{-\tau \Re \lambda }, 1\}}{M!} 
\end{equation}
The error bound \eqref{error_m_q} follows from $e_M = Vz_M$, using \eqref{eq_error_2}, \eqref{rm} and \eqref{bound_discr} for $\lambda \in \overline{B}$.
}
\end{proof}



\subsection{Characteristic roots} \label{s:conv_CE}

We first study the exponential solutions of \eqref{linear_ODEM} and specify the corresponding discrete characteristic equation.
Next, we prove the convergence of the roots of the characteristic equations and, in turn, of exponential solutions.
\begin{lemma}
\fra{
Let $M\in \mathbb{N}$ and let $\lambda \in \mathbb{C}$, $\lambda \notin \sigma(D_M)$.
Then \eqref{linear_ODEM} has a solution of the form $x(t)=\ee^{\lambda t}y$ with nontrivial $y \in \mathbb{C}^M$ if and only if $\lambda$ is a root of the characteristic equation
\begin{equation} \label{discr_CE}
1=K_M (\lambda I-D_M)^{-1} (-\mathbf{1}).
\end{equation}
}
\end{lemma}
\begin{proof}
Substitution of $\ee^{\lambda t}y$ into \eqref{linear_ODEM} leads to 
$$
\lambda y = D_M y-(K_My) \, \mathbf{1}.
$$
First note that, for $M \geq \overline M(\lambda)$, $K_My \neq 0$ must hold. Indeed $\lambda \notin \sigma(D_M)$ and therefore $\lambda y = D_M y$ can not hold.
Next, we can invert the equation to obtain
$$y= (K_M y)(D_M-\lambda I)^{-1}\mathbf{1}.$$
Hence we should have 
$$ K_M y = (K_M y)\, K_M (D_M-\lambda I)^{-1}\mathbf{1},$$
which amounts to \eqref{discr_CE} since $K_M y \neq 0$.

\smallskip
\fra{Vice versa, assume \eqref{discr_CE} holds and take $y := (D_M-\lambda I)^{-1}\mathbf{1}$. Note that $y \neq 0$ because $(D_M-\lambda I)$ has full rank. Moreover, $K_My = \mathbf{1}$, hence
$$ (D_M-\lambda I)y = \mathbf{1} = K_My,$$
i.e., \eqref{discr_CE} holds and therefore $\ee^{\lambda t}y$ is a solution of \eqref{linear_ODEM}.
}
\end{proof}

\bigskip
Let now
\begin{align*}
\chi(\lambda) &:= 1 - \int_0^\tau k(a) \ee^{-\lambda a} \diff a = 1 - \langle k, \psi_\lambda \rangle, \\
\chi_M(\lambda) &:= 1 - K_M (D_M-\lambda I)^{-1} \mathbf{1} = 1 - \langle k, P_M (\lambda I-D_M)^{-1} (-\mathbf{1}) \rangle, 
\end{align*}
with $\psi_\lambda$ defined in \eqref{psi_lambda},
so that the roots of \eqref{CE} and \eqref{discr_CE} coincide with the zeros of $ \chi(\lambda)=0$ and $\chi_M(\lambda)=0$, respectively.
We want to prove the convergence of the zeros of $\chi_M$ to the zeros of $\chi$.
Lemma \ref{l:res_1} allows us to state the following result.

\begin{theorem} \label{th:conv_CE}
\fra{Let $B$ be an open bounded subset of $\mathbb{C}$, and let \eqref{cond_Lebesgue} hold.
There exists $\overline M(B)$ such that, for any index $M \geq \overline M(B)$ and $\lambda \in B$, $\lambda \notin \sigma(D_M)$, 
$$ |\chi(\lambda)-\chi_M(\lambda)| \leq C_2(B) \frac{C_1(B)^M}{M!} \, \| k\|_\infty,$$
for $C_1(B)$ and $C_2(B)$ defined as in Lemma \ref{l:res_1}.
}
\end{theorem}
\begin{proof}
We have
\begin{align*}
\big|\chi(\lambda)-\chi_M(\lambda)\big| 
	&= \big| \langle k , \psi_{\lambda} \rangle - \langle k, P_M \fra{(\lambda I - D_M)^{-1}} \mathbf{1} \rangle \big| \\
	&\leq \| k \|_{\infty} \| \psi_\lambda - P_M  \fra{(\lambda I - D_M)^{-1}} \mathbf{1}\|_{NBV},
\end{align*}
\fra{and the assertion follows from \eqref{error_m_q}.}
\end{proof}
We are now ready to conclude that roots of \eqref{CE} are approximated by roots of \eqref{discr_CE}.
The proof of the following result follows the lines of \cite[Section 5.3.2]{BMVBook}. We repeat the proof here for completeness.

\begin{theorem} \label{th:conv_lambda}
Let $\lambda$ be a root of \eqref{CE} with multiplicity $\nu$, and let $B$ be an open ball of center $\lambda$ such that \fra{there is no other characteristic root of \eqref{CE} in $B$}.
Under assumption \eqref{cond_Lebesgue}, there exists a positive integer $\overline{M}=\overline{M}(B)$ such that, for $M \geq \overline{M}$, there exist $\nu$ roots $\lambda_1,\dots,\lambda_\nu$ (each counted as many times as its multiplicity) of \eqref{discr_CE} with
$$ \max_{j=1,\dots,\nu} \big| \lambda - \lambda_j \big| \leq  \left(  \frac{2 \nu!}{|\chi^{(\nu)}(\lambda)|} \max_{z\in B\setminus\{\lambda\}} |\chi(z)-\chi_M(z)| \right)^{\frac{1}{\nu}}.$$
In particular, $ \max_{j=1,\dots,\nu} | \lambda - \lambda_j | \to 0 $ as $M\to \infty$. 
\end{theorem}

\begin{proof}
We want to apply Rouch\'e's theorem from complex analysis, which states that, if $\chi$ and $\chi_M$ are continuous on a compact set $K$ and holomorphic in the interior of $K$, and $|\chi(z)-\chi_M(z)|<|\chi(z)|$ for all $z \in \partial K$, then $\chi$ and $\chi_M$ have the same number of zeros inside $K$, see \cite[Chapter 10]{RudinBookComplex} or \cite[Section 7.7]{PriestleyBook}.
Clearly, the complex-valued functions $\chi$ and $\chi_M$ are holomorphic in $B$.
Let
$$ \varepsilon_M(z) :=  \big| \chi(z) - \chi_M(z)\big|.$$
From Theorem \ref{th:conv_CE}, we have that $ \varepsilon_M(z) \to 0$ for all $z \in B\setminus \{\lambda\}$.

\medskip
Since $\lambda$ has multiplicity $\nu$, we have $\chi^{(n)}(\lambda) =0$ for $n=0,\dots,\nu-1$.
For $z$ in a neighborhood $U\subset B$ of $\lambda$, by Taylor expansion we have
\fra{
\begin{equation} \label{chi_taylor}
\chi(z) = \frac{\chi^{(\nu)}(\lambda)}{\nu!} (z-\lambda)^\nu + R_\nu(z),
\end{equation}
with
$$ R_\nu(z) = \sum_{j=\nu+1}^\infty \frac{\chi^{(j)}(\lambda)}{j!} (z-\lambda)^j = \sum_{j=\nu+1}^\infty \frac{(z-\lambda)^j}{2\pi i}  \int_{\partial B} \frac{\chi(w)}{(w-\lambda)^{j+1}} \diff w.
$$
If $r$ denotes the radius of the ball $B$ centered in $\lambda$, we have
\begin{equation} \label{R-bound}
|R_\nu(z)| \leq M_B \sum_{j=\nu+1}^\infty \frac{(z-\lambda)^j}{r^j} \leq M_B \frac{\alpha(z)^{\nu+1}}{1-\alpha(z)}
\end{equation}
where $M_B := \sup_{z \in \partial B} \frac{|\chi(z)|}{2\pi r}$ and $\alpha(z) = \frac{|z-\lambda|}{r}<1$.
Take now $r_1=r_1(B)$ such that $B(\lambda,r_1) \subset B$ and, for all $z \in B(\lambda,r_1)$,
\begin{equation} \label{r1}
M_B \frac{\alpha(z)^{\nu+1}}{1-\alpha(z)} < \frac{1}{2} \frac{|\chi^{(\nu)}(\lambda)|}{\nu!} |z-\lambda|^\nu.
\end{equation}
From \eqref{chi_taylor}, \eqref{R-bound} and \eqref{r1} we have that,
}
%
%
for all $z \in B(\lambda,r_1)$,
\begin{equation} \label{error_2}
|\chi(z)| > \frac{1}{2} \frac{|\chi^{(\nu)}(\lambda)|}{\nu!} |z-\lambda|^\nu.
\end{equation}
Now take $\overline M= \overline M(B)$ large enough so that, for all $M\geq \overline M$,
\begin{equation} \label{error_1}
\max_{|z-\lambda|=r_1} \varepsilon_M(z) \leq \frac{1}{2} \frac{|\chi^{(\nu)}(\lambda)|}{\nu!} r_1^\nu,
\end{equation}
and define $r^*=r^*(M)$ by
\begin{equation} \label{r*}
r^* := \left( \frac{\max_{|z-\lambda|=r_1} \varepsilon_M(z)}{\frac{1}{2} \frac{|\chi^{(\nu)}(\lambda)|}{\nu!}} \right)^{\frac{1}{\nu}}.
\end{equation}
Note that, by \eqref{error_1}, it is $r^*\leq r_1$.
Then, by combining \eqref{error_2} and \eqref{error_1}, we have
\begin{align*}
\max_{|z-\lambda|=r^*} \big| \chi(z) - \chi_M(z) \big| 
	&\leq \max_{|z-\lambda|=r_1} \big| \chi(z) - \chi_M(z) \big|  \\
	&= \frac{1}{2} \frac{|\chi^{(\nu)}(\lambda)|}{\nu!} (r^*)^\nu \\
	& < \fra{|\chi(z)|},
\end{align*} 
where the last inequality holds for all $z$ such that $|z-\lambda| = r^*$.
Hence for all $M\geq \overline M$ we can apply Rouch\'e's theorem on $B(\lambda,r^*)$, and conclude that $\chi$ and $\chi_M$ have the same number of zeros in $B(\lambda,r^*)$, where each zero is counted as many times as its multiplicity.

More precisely, there exist $\lambda_1,\dots,\lambda_\nu$ such that, for all $j=1,\dots,\nu$, $ \chi_M(\lambda_j)=0$ and
$$ |\lambda - \lambda_j| < r^*.$$
Now note from \eqref{r*} that $r^*=r^*(M) \to 0$ as $M\to \infty$. 
Hence we have proved the convergence of characteristic roots with their multiplicity.
\end{proof}

Thanks to the previous theorem we know that each ``true'' characteristic root is approximated as $M\to \infty$ with its multiplicity. 
\fra{Finally, given a sequence $\{\lambda_M\}_{M}$ satisfying $\chi_M(\lambda_M)=0$ for $M\geq \overline{M}$ for some $\overline M \in \mathbb{N}$, and such that $\lambda_M \to \lambda \in \mathbb{C}$ as $M\to \infty$, from the continuity of $\chi$ and $\chi_M$ we can conclude that $\chi(\lambda)=0$.
}

Combining the result of Theorem \ref{th:conv_lambda} with the bound in Lemma \ref{l:res}, it follows that the order of convergence of the characteristic roots depends on the interpolation error of $\psi_\lambda'$.
In particular, Cauchy's remainder theorem \cite[Theorem 3.1.1]{Davis1975} gives the more precise bound
\begin{equation} \label{bound_modulus}
\|(I-\mathcal{L}_{M-1}) \psi_\lambda' \|_\infty \leq \frac{\tau^{M} \| \psi_\lambda^{(M)}\|_{\infty}}{M!} \leq \frac{\tau^{M} |\lambda|^{M-1} \max\{ \ee^{- \tau \Re \lambda}, 1\} }{M!} .
\end{equation} 
Hence, the modulus of $\lambda$ affects the order of convergence of the roots of the discrete characteristic equation, as observed also in \cite{BMVBook,BMV2005}. More precisely, characteristic roots that are smaller in modulus are approximated with better accuracy for the same value of $M$. This can indeed be observed experimentally in Figure \ref{f:specialRE_eig}.

\subsection{Numerical results}

\begin{figure}[p]
\centering
\includegraphics[width=\textwidth]{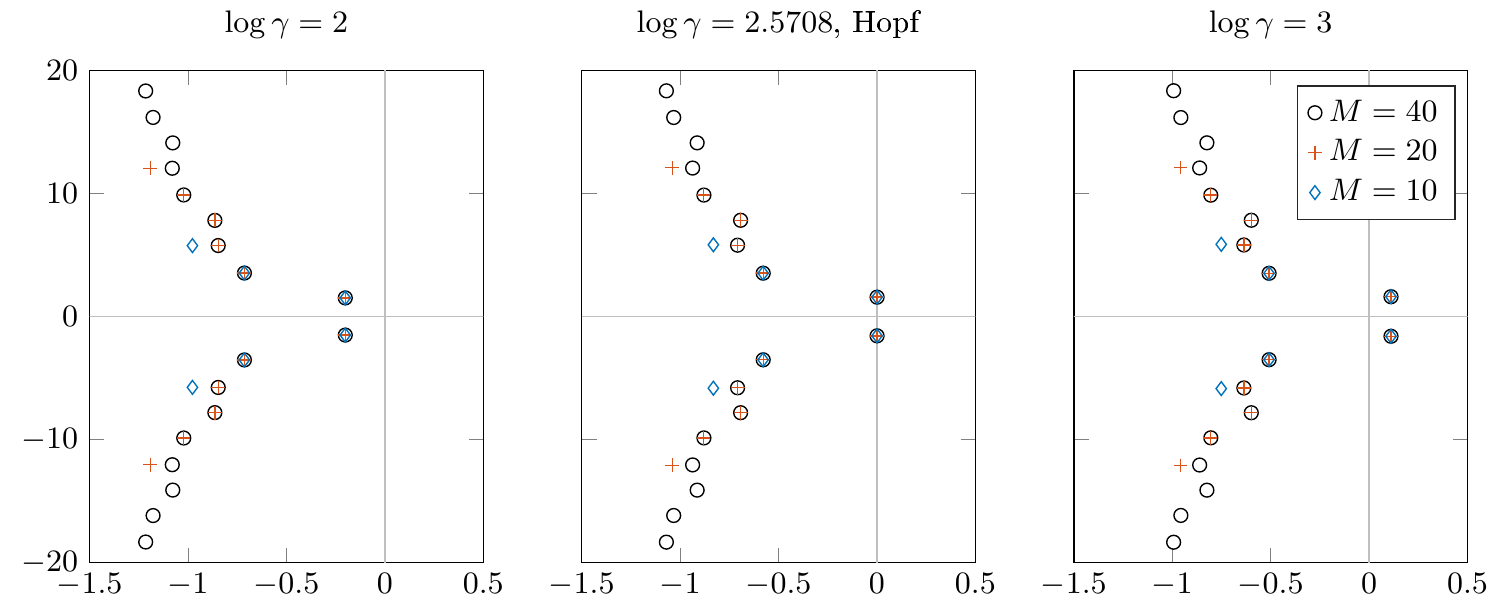} 
\caption{Equation \eqref{specialRE}. Eigenvalues associated with the positive equilibrium at $\log\gamma=2,2.57,3$, for $M=10,20,40$.
Larger (in modulus) eigenvalues to the left are not shown.
Note the rightmost pair of eigenvalues crossing the vertical axis in a Hopf bifurcation, as $\log\gamma$ increases.
}
\label{f:specialRE_eig}
\end{figure}

\begin{figure}[p]
\includegraphics[width=\textwidth]{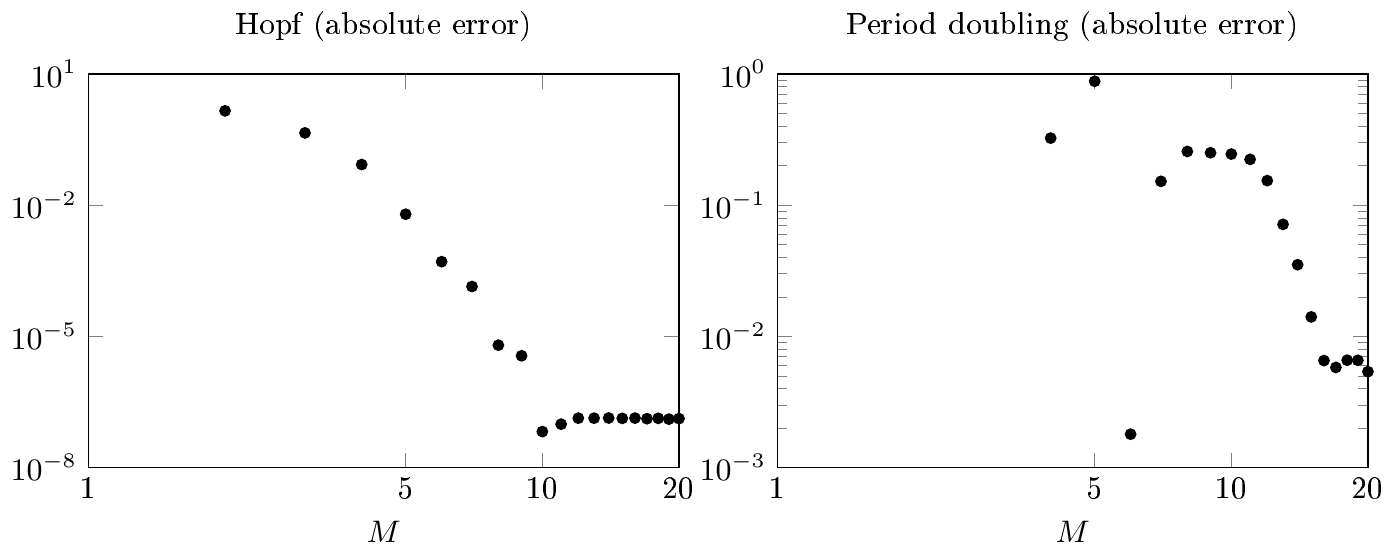}
\caption{Equation \eqref{specialRE}. Log-log plot of the error in the detection of the Hopf bifurcation point (left) and the first period doubling bifurcation (right), increasing $M$.
}
\label{f:specialRE_errors}
\end{figure}

\begin{figure}[t]
\centering
\includegraphics[width=\textwidth]{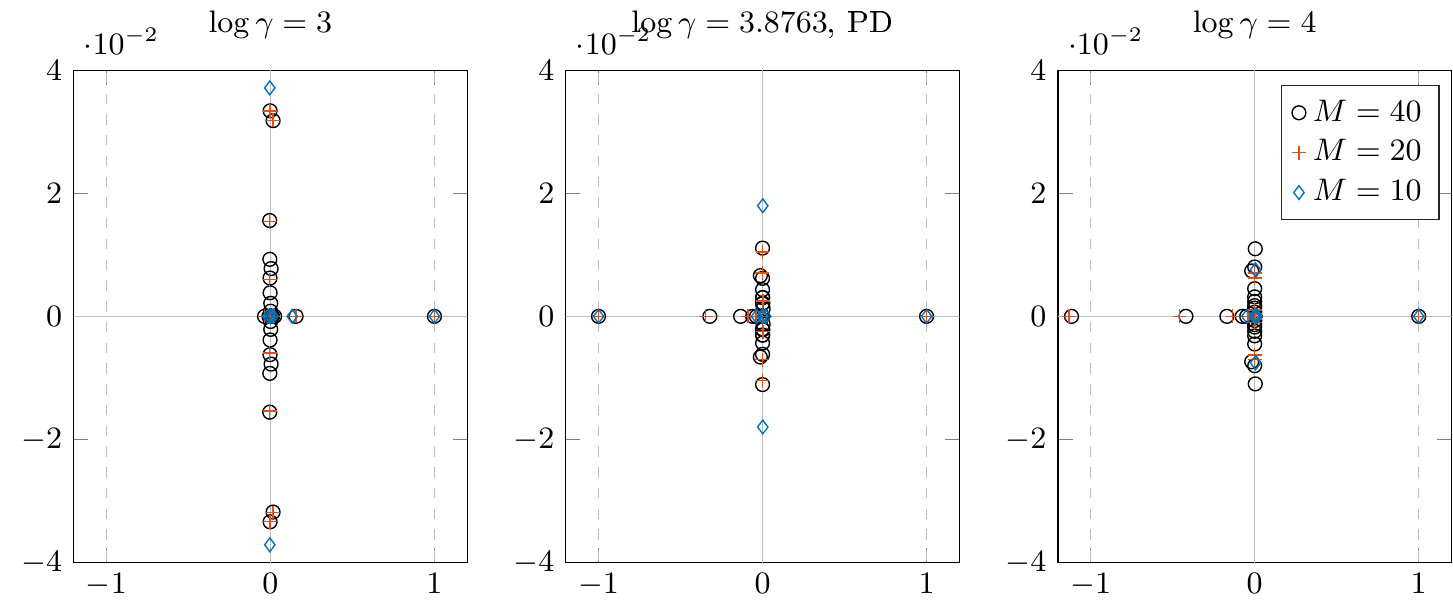}
\caption{
Equation \eqref{specialRE}. Multipliers associated with the branch of periodic solutions emerging from Hopf, at $\log\gamma=3,3.8763,4$, for $M=10,20,40$.
Note that the periodic solutions lose stability via a period doubling bifurcation, with the leftmost multiplier exiting the unit circle through $-1$ as $\log\gamma$ increases.
For $M=10$ and $\log\gamma=4$ the leftmost multiplier is at approximately $\lambda \approx -2.6793$ (not visible in the leftmost panel).
}
\label{f:specialRE_mult}
\end{figure}

To illustrate the convergence result proved in Theorem \ref{th:conv_lambda} we consider again equation \eqref{specialRE}, focusing on the approximation of the eigenvalues associated with the nontrivial equilibrium.
In the following analyses, we examine the approximated eigenvalues returned as output of the numerical continuation with MatCont and study experimentally the convergence as $M$ increases.
Similar analyses have been performed on the other examples presented in Section \ref{s:examples}, showing analogous results.
Explicit formulas for the differentiation matrix associated to \eqref{cheb_zeros} with $\theta_0=0$, as well as the barycentric weights for interpolation, are reported in \ref{app:cheb_zeros}.

Figure \ref{f:specialRE_eig} shows the approximated eigenvalues associated with the nontrivial equilibrium for $\log\gamma=2,3$, and at the detected Hopf bifurcation point, located at $\log\gamma \approx 2.5708$.
We plotted the approximated spectrum for several values of $M$, with $M=10,20,40$. 
Note that the rightmost eigenvalues are well approximated already for $M=10$, whereas eigenvalues that are larger in modulus require higher values of $M$ to obtain a satisfactory approximation. 
This is in agreement with the bound \eqref{bound_modulus}, since the modulus of $\lambda$ affects the speed of convergence.
The different panels show the rightmost pair of eigenvalues crossing the vertical axis from left to right in a Hopf bifurcation.

As an outlook on the approximation of bifurcation points of equilibria and periodic solutions, we have studied experimentally the error behavior of the MatCont bifurcation points, Hopf and period doubling, see Figure \ref{f:specialRE_errors} (log-log plot). 
The error is computed with respect to the values obtained with the current method and $M=40$.
Both plots show the typical spectral accuracy behavior, but a reliable approximation of the period doubling bifurcation requires a larger dimension of the approximating system. 
The barrier of $10^{-7}$ evident in the left panel is likely due to the tolerance options imposed in MatCont (tolerance $10^{-10}$ for Newton's method and $10^{-6}$ for the calculation of the test functions for bifurcation points).
The computation of the period doubling was performed with tolerance options $10^{-6}$ in MatCont.
Although we did not study theoretically the convergence of multipliers, we can expect that, similarly as in Theorem \ref{th:conv_lambda} for eigenvalues, the multiplicity plays a role in the order of convergence.
The fact that the trivial multiplier $\mu=1$ has multiplicity 2 due to the integration of the state may in turn have consequences on the convergence rate of the multipliers corresponding to periodic solutions.

Regarding periodic solutions, we plot the approximated multipliers associated with the branch of periodic solutions emerging from Hopf, for $\log\gamma=3,4$, and at the period doubling bifurcation detected at $\log\gamma \approx 3.8763$, see Figure \ref{f:specialRE_mult}. 
Note the leftmost multiplier exiting the unit circle via $-1$, indicating a period doubling bifurcation.

\section{Discussion and outlook} \label{s:outlook}

We have proposed a new approach for the approximation of a nonlinear RE with a system of ODE.
Similarly as in \cite{SIADS2016,EJQTDE2016}, the approach is based on the formulation of the problem as an abstract Cauchy problem on a space of functions and its approximation via pseudospectral techniques.
The novelty here is that the space of functions is chosen as the space $AC_0$, rather than $L^1$. This is done by first integrating the state and then considering the integrated variable as state variable. A similar idea has been proposed for PDE in \cite{Vietnam2020}.
Compared with the method proposed in \cite{SIADS2016}, this approach returns an approximation that is naturally formulated as a system of ODE, avoiding the algebraic equation emerging from the rule for extension of the RE. 
In terms of efficiency, the reduction in computational cost is approximately tenfold, as evidenced in Figure \ref{f:specialRE_time} and Table \ref{t:times}.

The attention here is restricted to the subspace $\fra{AC_0} \subset NBV$ consisting of absolutely continuous functions. In \cite{TwinSemigroups}, the variation-of-constants formula and the construction of the solution operators of linear equations are in fact extended to $NBV$. So one might wonder whether pseudospectral approximation can be extended from $\fra{AC_0}$ to $NBV$? In the $NBV$ setting, solutions are allowed to make jumps when time proceeds. 
To capture the jumps  of $v(t)$, the discretized vector $x(t)$ should be allowed to have jumps as well. As a consequence, one cannot in general describe the dynamics of the vector in terms of just a differential equation. 
In this case one could consider other techniques like spectral or finite elements approximations.
%

Although the construction is carried out in the space $NBV$, we resorted to bounds of the interpolation error in supremum norm (hence in a subspace of $C$) when proving the error bounds in Lemma \ref{l:res}, since the literature about convergence of interpolation in supremum norm is traditionally more extensive than the corresponding one in the $L^1$-norm. We wonder if similar (or even stronger) conclusions can be stated in terms of the $NBV$ norm, possibly exploiting results about the interpolation of $NBV$ functions and bounds in terms of bounded variation norm \cite{MastroianniBook, Trefethen2013, Mastroianni2010}.

In our theoretical results, the condition \eqref{cond_Lebesgue} concerning the asymptotic (for the number of points going to infinity) behavior of the Lebesgue constant plays a crucial role. It is known that for Chebyshev zeros, cf.~\eqref{cheb_zeros}, the condition is satisfied, see \cite[Chapter 1.4.6]{MastroianniBook}. For other meshes the condition is not guaranteed. For instance, for Chebyshev extremal nodes without one or both endpoints it is only known that the quantity at the left-hand side of \eqref{cond_Lebesgue} is bounded for $M$ tending to infinity \cite[Chapter 4.2]{MastroianniBook}.
However, Chebyshev extrema are widely used for pseudospectral differentiation and integration in the bounded interval, and efficient numerical routines exist for the associated differentiation matrix, and interpolation and quadrature weights \cite{Trefethen2000,Weideman2000}.
Our experimental results (not included here), show comparable results when using Chebyshev zeros or extrema, both in terms of accuracy and computational times. It would be interesting if either \eqref{cond_Lebesgue} could be verified for a large class of meshes, or the proof of convergence could be restructured such that only a weaker variant of \eqref{cond_Lebesgue}, known to hold for a large class of meshes, is needed.

We proved that every characteristic root of a linear(ized) RE is approximated by the characteristic roots of the corresponding pseudospectral approximation when the dimension $M$ is large enough.
Vice versa, the limit of every convergent sequence of discrete characteristic roots is a ``true'' characteristic root of the delay equation. An open problem is proving that the dimension of the unstable manifold is preserved for $M$ large enough, or alternatively that the infinite- and finite-dimensional problems have the same number of eigenvalues lying in any right-half of the complex plane, if $M$ is large enough. 

We also remark that proving convergence of Hopf bifurcations requires not only the convergence of the eigenvalues, but also to verify that the transversality conditions are satisfied at the bifurcation point, and the direction of bifurcation is preserved as $M\to \infty$ \cite{DGG2007,DiekmannBook}. 
For DDE, this is done in \cite{Babette}. To adapt the proofs to RE, one should obtain explicit formulas for the direction of Hopf and verify the convergence (see \cite[Remark 2.22]{DGG2007} and \cite{Babette}). 

The natural next step beyond the study of equilibria is the approximation of periodic solutions and their convergence as $M\to \infty$. In this case, one should first ensure the convergence of the finite-time solution maps, and hence focus the attention on the initial value problems associated with the RE and its ODE approximation.
A study of the convergence of the solution operators is in the authors' pipeline.
The approximation of stability of the periodic solutions then relies on the approximation of the multipliers of the (time periodic) linearized equations, as formally proved in \cite{Breda2020}.

The analysis of Nicholson's blowflies equation, and in particular Figure \ref{f:Nich-fixed-M}, shows the impossibility of using a discretization of a truncated interval to approximate the behavior of an equation with infinite delay. In this case, techniques specific to the unbounded integration interval seem necessary, along the lines of \cite{AMC2018, IlariaThesis}. 

To keep the notation simple, in Section \ref{s:method} we introduced the approximation approach for scalar equations. We stress however that the extension to systems of equations is quite straightforward and can be done along the lines of \cite{SIADS2016}. In fact, the combination of the current method for RE with the pseudospectral discretization of DDE \cite{SIADS2016,BredaSanchez2015} provides a strategy for approximating general systems  where a RE is coupled with a DDE, which arise frequently in population dynamics, see e.g.\ \cite{DGM2010}.

\section*{Acknowledgements}
\noindent
\fra{The authors are thankful to two anonymous reviewers for their constructive comments that improved the manuscript.
The research of FS was supported by the NSERC-Sanofi Industrial Research Chair in Vaccine
Mathematics, Modelling and Manufacturing.
FS and RV are members of the INdAM Research group GNCS and of the UMI Research group ``Modellistica Socio-Epidemiologica''.}

{ \footnotesize

}

\appendix
\section{Renewal equation formulation of Nicholson's blowflies equation}
\label{app:blowflies}

\begin{proposition}
\fra{Let $c>0$ and
\begin{equation} \label{b-blowflies}
b(t):= \frac{\beta_0}{c} A(t) \ee^{-A(t)}.
\end{equation}
}
Then \eqref{Nich-DDE} is equivalent to $b(t)=F(b_t)$ with
\begin{equation} \label{F-Nich}
F(\phi) := \beta_0 \int_1^{\infty} \phi(-s) \mathcal{F}(s) \diff s \, \ee^{- c \int_1^{\infty} \phi(-s) \mathcal{F}(s) \diff s}
\end{equation}
with $\mathcal{F}(s)=\ee^{-\mu s}$.
\end{proposition}

\begin{proof}
The equation $A'(t)=-\mu A(t) + c \ee^{-\mu} b(t-1)$ implies 
\begin{align*}
A(t) &= c \int_{-\infty}^t \ee^{-\mu(t-s)} \ee^{-\mu} b(s-1) \diff s \\
	&= c \int_{-\infty}^{t-1} \ee^{-\mu(t-\sigma)} b(\sigma) \diff \sigma \\
	&= c \int_1^{\infty} \ee^{-\mu \eta} b_t(-\eta) \diff \eta.
\end{align*}
Now substitute this expression for $A(t)$ in \eqref{b-blowflies}.
\end{proof}

\fra{Equation \eqref{Nich-RE} is obtained from \eqref{F-Nich} with $c=1$. However, it turned out that, for the numerical continuation with MatCont, the choice $c=100$ worked better than $c=1$. So the Figures \ref{f:Nich-fixed-tau} and \ref{f:Nich-fixed-M} were made using $c=100$.
}

\section{Chebyshev zeros with one additional endpoint}
\label{app:cheb_zeros}

We include here some useful formulas for the construction of the differentiation matrix and the barycentric weights associated with the mesh $\Theta_M \cup \{0\}$ used for the numerical simulations in the paper.
We first summarise the formulas for the classical Chebyshev zeros in the interval $(-1,1)$ with the addition of the endpoint $1$, and then show how to obtain the corresponding formulas for the nodes shifted in the delay interval $(-\tau,0]$. 

The Chebyshev zeros are defined as the roots of the Chebyshev polynomial of the first kind of degree $M$, see e.g. \cite{Xu2016} for a recent review.
In the interval $(-1,1)$, the nodes are given by the explicit formulas
\begin{equation}\label{cheb_zeros_app}
x_j = \cos \left( \frac{2j-1}{2M}\pi \right), \qquad j=1,\dots,M,
\end{equation}
or the equivalent expression
\begin{equation*}
x_j= \sin \left(\frac{M-(2j-1)}{2M}\pi \right), \qquad j=1,\dots,M,
\end{equation*}
which maintains the exact symmetry of $x_j$ about the origin in floating-point arithmetic \cite{Xu2016}.

Let 
$$\ell(x) := \prod_{j=1}^M (x-x_j)$$
be the node polynomial. 
Using barycentric weights, defined by 
$$ w_j = \frac{1}{\prod_{k\neq j} (x_j-x_k)} = \frac{1}{\ell'(x_j)}, \qquad j=1,\dots,M,$$
the polynomial $p$ interpolating the points $y_j$ on the nodes $x_j$, $j=1,\dots,M$, is efficiently computed by the barycentric interpolation formula
\begin{equation} \label{barycentric}
p(x) = \frac{\sum_{j=1}^M \frac{w_j}{x-x_j} y_j}{\sum_{j=1}^M \frac{w_j}{x-x_j}}
\end{equation}
see, e.g., \cite{Berrut2004}. We recall that the main computational advantage of \eqref{barycentric} is that $p(x)$ can be evaluated in $O(M)$ floating-point operations, once the barycentric weights $w_j$ are given.

For the nodes \eqref{cheb_zeros_app}, the weights $w_j$ are given by the explicit formulas
\begin{equation} \label{wj}
w_j = \frac{2^{M-1}}{M} (-1)^j \sin \left( \frac{2j-1}{2M}\pi \right), \qquad j=1,\dots,M.
\end{equation}
It is important to note that the factor $ \frac{2^{M-1}}{M}$, independent of $j$, \fra{cancels out} in the numerator and denominator of \eqref{barycentric}, so it can be ignored for computational reasons. We decided to include it explicitly in \eqref{wj} because it will become important with the later addition of the node $x=1$.

To compute the differentiation matrix $(d_{kj})_{j,k=1,\dots,M}$, we first define
$$ c_j := \prod_{k\neq j} (x_j-x_k) = w_j^{-1}.$$
The off-diagonal entries are computed by
$$ d_{kj} = \frac{1}{x_k-x_j} \frac{c_k}{c_j} \qquad k \neq j
$$
and the diagonal entries $d_{kk}$, $k=1,\dots,M$ can be computed for instance by imposing that $\sum_{j=1}^M d_{kj}=0$ for all $k=1,\dots,M$, see, e.g., \cite{Trefethen2000,Weideman2000}.

\bigskip
We now construct the corresponding elements $\ell^+_j$, $w^+_j$, $d^+_{kj}$ for the set of $M+1$ nodes formed by \eqref{cheb_zeros_app} with the addition of $x_0 = 1$. 
Define 
$$K := \ell(1)^{-1} = \frac{1}{\prod_{k=1}^M (1-x_k)},$$
then \fra{the Lagrange polynomials become}
$$
\ell^+_0(x) = \prod_{k=1}^M \frac{x-x_k}{1-x_k} = K \ell(x), \qquad
\ell^+_j(x) = \prod_{\substack{k=0 \\ k\neq j}}^M \frac{x-x_k}{x_j-x_k} = \frac{x-1}{x_j-1} \ell_j(x), \qquad j=1,\dots,M,
$$
and, for the barycentric weights, we have
$$
w^+_0 = K, \qquad
w^+_j = \frac{w_j}{x_j-1} = \frac{2^{M-1}}{M} \frac{(-1)^j}{x_j-1} \sin \left( \frac{2j-1}{2M}\pi \right), \qquad j=1,\dots,M.
$$
It is now clear that, when simplifying the factor $\frac{2^{M-1}}{M}$ from $w_j$ and $w^+_j$, $j=1,\dots,M$, we should also normalize the weight $w^+_0$.
Next we have, for $c^+_j := (w^+_j)^{-1}$, 
$$ d^+_{kj} = \frac{1}{x_k - x_j} \frac{c^+_k}{c^+_j}, \qquad k\neq j.
$$
and the diagonal coefficients $d^+_{kk}$ are then normalized as before by imposing $\sum_{j=0}^M d^+_{kj}=0$ for all $k=0,\dots,M$.
To approximate integrals over the interval $(-1,1)$ we can use quadrature formulas corresponding to the nodes \eqref{cheb_zeros_app}, with corresponding weights
$$ q_k = \frac{2}{M} \left( 1- 2 \sum_{j=1}^{\lfloor M/2 \rfloor} \frac{\cos 2 j \xi_k}{4j^2-1} \right), \qquad j=1,\dots,M, $$
where $\xi_k = \frac{(2k-1)\pi}{M}$ are the arguments of the cosine in \eqref{cheb_zeros_app}, see \cite{Xu2016}.
We also recall that, to avoid cancellations when computing the differences $x_k-x_j$ in the differentiation matrix, one can use the trigonometric identities (see \cite{Weideman2000})
\begin{equation*}
x_k - x_j = \cos \xi_k - \cos \xi_j = 2 \sin \left(\frac{\xi_k+\xi_j}{2}\right) \, \sin\left(\frac{\xi_j-\xi_k}{2}\right).
\end{equation*}

\bigskip
Finally, we can construct the desired mesh in the interval $(-\tau,0]$ by simply taking $\theta_0=0$, and shifting and scaling the nodes $x_j$ as
$$ \theta_j = \frac{\tau}{2} (x_j-1), \qquad j=0,\dots,M.$$
The corresponding barycentric weights $w^\tau_j$ are obtained by a suitable scaling factor via
$$ w^\tau_j = \frac{1}{\prod_{k\neq j} (\theta_j-\theta_k)} 
= \left(\frac{2}{\tau}\right)^M w^+_j, \qquad j=0,\dots,M.
$$
Note again that the multiplicative factor, independent of $j$, can be ignored as it cancels out in \eqref{barycentric};
finally the entries of the differentiation matrix $(d^\tau_{kj})_{k,j=0,\dots,M}$ and the quadrature weights $q^\tau_k$ satisfy
$$ d^\tau_{kj} = \frac{2}{\tau} d^+_{kj}, \qquad q^\tau_k = \frac{\tau}{2} q_k, \qquad k,j=0,\dots,M. $$

%
%

\end{document}